\documentclass[reqno,11pt]{article}

\usepackage{a4wide}
\usepackage{geometry}
\usepackage{hyperref}
\usepackage{graphicx}
\usepackage{subfigure}
\usepackage[toc,page]{appendix}
\usepackage{xcolor}
\usepackage{pst-all}
\usepackage[frenchb, english]{babel}
\usepackage{pdfsync}
\usepackage{amsmath,amssymb,amsthm}
\usepackage{mdwlist}
\usepackage{mathrsfs}
\usepackage{mathtools}
\usepackage{bm}
\usepackage{dsfont}
\usepackage{bbm}

\usepackage{tikz}
\usepackage[framemethod=TikZ]{mdframed}

\DeclareMathOperator{\diam}{diam}
\DeclareMathOperator{\supp}{supp}
\newcommand{\norm}[1]{\left\lVert#1\right\rVert}

\newcommand{\Leb}{\mathscr{L}}

\newcommand{\setN}{\mathbb{N}}
\newcommand{\N}{\mathbb{N}}
\newcommand{\R}{\mathbb{R}}

\newcommand{\X}{\mathsf{X}}
\newcommand{\p}{\mathtt p} 
\newcommand{\de}{\ensuremath{\, \mathrm d}} 
\newcommand{\dee}{\ensuremath{\mathrm d}}

\newcommand{\suchthat}{\ensuremath{\,:\,}} 
\newcommand\restr[2]{{
  \left.\kern-\nulldelimiterspace 
  #1 
  \right|_{#2} 
  }}

\newcommand{\weakto}{\rightharpoonup}

\newcommand{\CD}{\mathsf{CD}}

\DeclareMathOperator{\res}{restr}
\newcommand{\rest}[2]{{{\rm restr}_{#1}^{#2}}}
\newcommand{\di}{\mathsf d} 
\newcommand{\m}{\mathfrak m} 

\DeclareMathOperator{\Geo}{Geo}
\newcommand{\OptPlans}{\mathsf{Opt}}
\newcommand{\OptGeo}{\mathsf{OptGeo}}
\newcommand{\Prob}{\mathscr{P}}
\newcommand{\ProbTwo}{\mathscr{P}_2}
\newcommand{\ppi}{{\mbox{\boldmath$\pi$}}}

\usepackage{titlesec}
\titleformat{\section}
  {\scshape\large\centering}
  {\thesection}{0.5em}{}
\titleformat{\subsection}
  {\scshape\centering}
  {\thesubsection}{0.4em}{}

\numberwithin{equation}{section}

\newtheoremstyle{remark}
        {10pt}
        {10pt}
        {}
        {}
        {\itshape}
        {.}
        {.4em}
        {}

\newtheoremstyle{proof}
        {10pt}
        {10pt}
        {}
        {}
        {\itshape}
        {.}
        {.4em}
        {}
        
\newtheoremstyle{definition}
        {10pt}
        {10pt}
        {}
        {}
        {\bfseries}
        {.}
        {.4em}
        {}

\newtheoremstyle{theorem}
        {10pt}
        {10pt}
        {\slshape}
        {}
        {\bfseries}
        {.}
        {.4em}
        {}

\theoremstyle{theorem}
\newtheorem{theorem}{Theorem}[section]
\newtheorem{prop}[theorem]{Proposition}
\newtheorem{corollary}[theorem]{Corollary}
\newtheorem{lemma}[theorem]{Lemma}

\theoremstyle{definition}
\newtheorem{definition}[theorem]{Definition}

\theoremstyle{remark}
\newtheorem{remark}[theorem]{Remark}

\theoremstyle{proof}
\newtheorem*{pro}{Proof}
\newenvironment{pr}{\begin{pro}%
 \pushQED{\qed}}%
 {\popQED\end{pro}}

\title{Optimal maps and local-to-global property in negative  dimensional spaces
with Ricci curvature bounded from below}
\author{Mattia Magnabosco\thanks{Institut f\"ur Angewandte Mathematik, Universit\"at Bonn. Email: magnabosco@iam.uni-bonn.de}, Chiara Rigoni\thanks{Institut f\"ur Angewandte Mathematik, Universit\"at Bonn. Email: rigoni@iam.uni-bonn.de}}

\makeindex
\begin{document}
\maketitle

\begin{abstract}
In this paper we investigate two important properties of metric measure spaces satisfying the reduced curvature-dimension condition for negative values of the dimension parameter: the existence of a transport map between two suitable marginals and the so-called local-to-global property. 
\end{abstract}

\tableofcontents

\section{Introduction}

The class of metric measure spaces satisfying the $\CD(K, N)$-condition for $N < 0$ was first introduced by Ohta in \cite{Ohta16}, where the range of admissible ``dimension parameters'' in the theories of $(K, N)$-convex functions and of $\CD(K,N)$-spaces was extended to negative values. A further study of this family of spaces has been carried out in the recent work \cite{MRS21}, where a new setting to introduce the curvature-dimension condition for negative values of $N$ was proposed, extending and complementing the work by Ohta. Motivated by some of the examples shown in \cite[Section 3.2]{MRS21} (e.g., the interval $I := [-\pi/2, \pi/2]$ equipped with the Euclidean distance and the weighted measure $\de \m(x) := \cos^N(x) \de\mathcal{L}^1(x)$, $\mathcal{L}^1$ being the 1-dimensional Lebesgue measure on $I$, is a $\CD(N, N)$ space for any $N < -1$), the structures on which the study is focused are complete and separable metric spaces endowed with quasi-Radon measures, i.e., measures which are Radon outside a negligible set of singular points in which the measure can explode. In particular, this setting is more general than the classical one where the notion of curvature-dimension bounds  is introduced for non-smooth structures (see \cite{Sturm06I}, \cite{Sturm06II}, \cite{Lott-Villani09}): in fact, in the case in which the dimensional parameter is positive, the $\CD$-condition is introduced in complete and separable metric spaces equipped with Radon measures. In this regard, it is important to recall that the class of  metric measure spaces satisfying the $\CD(K, N)$-condition for some $N < 0$  includes all $\CD(K, \infty)$ spaces and so also all the  $\CD(K, N)$ ones for any $N > 0$. Furthermore, in \cite{Ohta16} the definition of reduced curvature-dimension condition was extended to the setting of negative values of the dimensional parameter. This condition, denoted by $\CD^\ast(K, N)$ and first introduced by Bacher-Sturm in \cite{BacherSturm10} for $N \ge 1$, is obtained from the $\CD(K, N)$-condition by replacing the volume distortion coefficients $\tau_{K, N}^{(t)}( \cdot )$ by the slightly smaller coefficients $\sigma_{K, N}^{(t)}( \cdot )$. Analogously to the result for $N \ge 1$ (see \cite[Proposition 2.5(i)]{BacherSturm10}), also in the case when $N <0$ the original curvature-condition $\CD(K, N)$ implies the reduced one $\CD^\ast(K, N)$, as proven in \cite[Proposition 4.7]{Ohta16}.\\

In this paper we prove two important properties of metric measure spaces satisfying the reduced curvature-dimension condition for negative values of the dimension parameter: the existence of a transport map between two suitable marginals and the so-called local-to-global property. Their validity for positive dimensional $\CD$ spaces has been established in a series of works that we are going to recall in the following. It is important to underline that in this setting the proofs of these two properties  rely heavily  on the lower semicontinuity of the characterizing entropy functional, which allows to perform some nice approximation arguments. Unfortunately in the context of quasi-Radon measures the lower semicontinuity of the entropy functional does not hold on the whole $\mathscr P_2(\X)$: instead our proofs are based on suitable extensions of these classical arguments, in which we have to pay particular attention to the set of singular points of the reference measure.\\

The problem of addressing existence and/or uniqueness of optimal transport maps between two given marginals has a long history, since it represents the original formulation of the optimal transport problem by Monge. The first positive answers were given by Brenier \cite{Brenier87} in the Euclidean setting, by McCann \cite{McCann01} in the setting of Riemannian manifolds, by Ambrosio-Rigot \cite{AR04} and Figalli-Rifford \cite{FR10} in the context of sub-Riemannian manifolds and by Bertrand \cite{Ber08} for Alexandrov spaces. In the context of metric measure spaces, most of the results are proven under a non-branching assumption and a metric curvature bound. In particular, we recall the works by Gigli \cite{Gigli12a}, Rajala-Sturm \cite{RajalaSturm12}, Gigli-Rajala-Sturm \cite{GigliRajalaSturm13} and Cavalletti-Mondino \cite{CM17}. Some results can also be obtained once the space is assumed to be non-branching and a quantitative property on the reference measure, related to the shrinking of sets to points, holds true (see the work by Cavalletti-Huesmann \cite{Cavalletti-Huesmann13} and the one by Kell \cite{Kell17}). As shown by Rajala in \cite{Rajala16}, the non-branching assumption is necessary to prove the uniqueness of the transport map. Nevertheless, some existence results can be proven also without this assumption, see for example the paper by Shultz \cite{Sch18} and by the first author \cite{Mag21}.

In this paper  we focus on the case of non-branching spaces satisfying the $\CD^\ast(K, N)$-condition for some $N < 0$ and we prove the existence and uniqueness of the optimal transport map when the marginals have finite entropy and bounded support. We then use this uniqueness result in order to show the existence of a transport map between two general absolutely continuous marginals with possibly unbounded support.\\

Also the local-to-global property of the curvature-dimension condition is a very important and fundamental feature. In fact, it shows that the metric measure space version of curvature-dimension bounds is a local requirement, as it happens in the case of Ricci curvature bounds in Riemannian manifolds. In the positive dimensional case, the local-to-global property was proven by Sturm \cite[Theorem 4.17]{Sturm06II} for the $\CD(K,\infty)$-condition, by Villani \cite{Villani09} for the $\CD(0, N)$-condition and then by Bacher-Sturm \cite[Theorem 5.1]{BacherSturm10} for the general $\CD^\ast(K,N)$ one. Thereafter, the globalization of the $\CD(K,N)$ condition was proven by Cavalletti-Milman \cite{CavMil16} with a much more sophisticated argument, which allowed them to demonstrate other remarkable properties of the $\CD$ condition.  We remark that a result similar to the one in \cite{CavMil16} seems out of reach in the context of $\CD(K, N)$ spaces with $N < 0$, due to the pathologies of the reference measure: what we prove in the last section of this paper is the equivalence between the local version of the $\CD^\ast(K-, N)$-condition and the global one, provided that the metric measure space $(\X, \di, \m)$ is locally compact. The main reason why we need to work with the $\CD^\ast(K-, N)$-condition, which consists in asking for the validity of the $\CD$ condition for any $K' \le K$, is to gain some flexibility that makes possible to find suitable discrete approximation arguments. In contrast to the positive dimensional case, in general, the $\mathsf{CD}^*(K-,N)$-condition in this setting is not equivalent to the $\mathsf{CD}^*(K,N)$ one. This is basically due to the fact that, in the negative dimensional case, the reference entropy functional is not lower semicontinuous in the Wasserstein space. However this equivalence holds when $K \ge 0$, which in particular ensures the validity of the local-to-global property for $\CD^\ast(K, N)$-spaces with $N < 0 \le K$.\\

\noindent{\bfseries Acknowledgements:} The second author gratefully acknowledges support by the European Union through the ERC-AdG 694405 RicciBounds. 

\section{Setting and preliminary results}

\subsection{Metric spaces and Wasserstein distance}

In this paper, $(\X, \di)$ will always denote a complete and separable metric space. The set $\mathscr P_2(\X)$ is the set of probability measures with finite second moment, that we endow with the Wasserstein distance $W_2$, defined by
\begin{equation}\label{eq:defW2}
    W_2^2(\mu, \nu) := \inf_{\pi \in \mathsf{Adm}(\mu,\nu)} \int \di^2(x, y) \, \de \pi(x, y),
\end{equation}
where $\mathsf{Adm}(\mu, \nu)$ is the set of all the admissible transport plans between $\mu$ and $\nu$, namely all the measures in $\mathscr P(\X^2)$ such that $(\p_1)_\sharp \pi = \mu$ and $(\p_2)_\sharp \pi = \nu$. The set of optimal transport plans between $\mu$ and $\nu$, that is the set of the admissible plans which realize the infimum in \eqref{eq:defW2}, is denoted by $\OptPlans(\mu,\nu)$.

In the following $C([0, 1], \X)$ will stand for the space of continuous curves from $[0, 1]$ to $\X$, endowed with the sup-norm, and we recall that this  is a complete and separable space. Hence, for any $t \in [0, 1]$ we define the evaluation map $e_t \colon C([0, 1], \X) \to \X$ by setting $e_t(\gamma) := \gamma_t$ and the stretching/restriction operator $\rest{r}{s}$ in $C([0, 1], \X)$, defined, for all $0\leq r<s\leq1$, by
\[
[\rest{r}{s}(\gamma)]_t := \gamma_{r + t(s-r)}, \qquad t \in [0, 1].
\]
A curve $\gamma \colon [0, 1] \to \X$ is a (minimizing constant speed) geodesic if
\[
\di(\gamma_s, \gamma_t) = |t-s| \di(\gamma_0, \gamma_1) \quad \text{for every }s,t\in[0,1],
\]
we indicate by $\Geo(\X)$  the space of geodesics on $\X$, endowed with the sup-norm. Observe that $\Geo(\X)$ is complete and separable as soon as $(\X, \di)$ satisfies the same properties.

In this terminology, if $\mu, \nu \in \mathscr P_2(\X)$ are joined by a geodesic, then their $W_2$-distance can be equivalently characterized as
\[
W_2^2(\mu, \nu) = \min_{\ppi} \int_0^1 |\dot{\gamma_t}|^2 \, \de t \de \ppi(\gamma),
\]
the minimum being taken among all $\ppi \in \mathscr P(C([0, 1], \X))$ such that $(e_0)_\sharp \ppi=\mu$ and $(e_1)_\sharp \ppi = \nu$. The set of minimizers will be denoted by $\OptGeo(\mu, \nu)$ and its elements will be called optimal geodesic plans. We underline that optimal geodesic plans are always supported in the set $\Geo(\X)$ and that a curve $(\mu_t)_{t\in[0,1]}$ is a geodesic connecting $\mu$ and $\nu$ if and only if there exists $\ppi \in \OptGeo(\mu, \nu)$ such that
\[
\mu_t = (e_t)_\sharp \ppi.
\]

\subsection{Metric spaces equipped with quasi-Radon measures}

In this section we present the setting we are going to work in. In order to motivate it, we recall that in \cite{MRS21} some new 1-dimensional model space have been provided, such as:
\begin{enumerate}
    \item the weighted space
\[
\Big(\R, |\cdot|, V \Leb^1\Big) \quad \text{ with } V(x) = \sinh\bigg(  x \sqrt{- \dfrac{K}{N}} \bigg)^N,
\]
which satisfies the the curvature-dimension condition $\CD(K, N+1)$ for any $K > 0$ and  $N < -1$;
\item the space $(\R, |\cdot|, |x|^{N} \Leb^1)$, which is a $\CD(0, N+1)$-space  for any $N < -1$;
\item the weighted space
\[
\Bigg(\Bigg[ - \dfrac{\pi}{2} \sqrt{\dfrac{K}{N}},  \dfrac{\pi}{2} \sqrt{\dfrac{K}{N}} \Bigg], |\cdot |,  \cos \bigg(  x \sqrt{\dfrac{K}{N}}\bigg) ^N \, \Leb^1 \Bigg)
\]
which is a $\CD(K, N+1)$-space for any $K < 0$  and $N < -1$.
\end{enumerate}
These examples show in particular that, in negative dimensional CD spaces, the reference measure does not need to be locally finite. Therefore we introduce the class of quasi-Radon measures, specializing to the case of measures defined on a complete and separable metric space. We refer to \cite[Section 2.1]{MRS21} for the precise definitions and constructions, presented in the more general setting of topological measure spaces.

\begin{definition}[Quasi-Radon measures]
Let $\m$ be a Borel measure defined on a complete and separable metric space $(\X, \di)$. We say that $\m$ is a quasi-Radon measure if it is complete and effectively locally finite, meaning that for every Borel $A\subset \X$ with $\m(A)>0$, there exists an open set $U\subset \X$ with finite measure such that  $\m(A\cap U)>0$.
\end{definition}
In particular, in \cite[Proposition 2.7]{MRS21}, it is proven that any quasi-Radon measure $\m$ defined on a complete and separable metric space $(\X, \di)$ is $\sigma$-finite and has the following property:
\begin{equation}\label{def:Sm}
\begin{split}
    \text{there exists} & \text{ a closed set } \mathcal S_\m \subset \X \text{ with empty interior and } \m(\mathcal S_\m) = 0\\ &\text{ such that } \m|_{\X \setminus \mathcal S_\m} \text{ is a Radon measure on } \X \setminus \mathcal S_\m.
\end{split}
\end{equation}
Intuitively, the set $\mathcal S_\m$ consists of all points $x$ such that $\m(U) = \infty$ for any open neighborhood $U$ of $x$ in $\X$. In the following, we will refer to $\mathcal S_\m$ as the singular set associated to $\m$.

We recall that in this setting it is still possible to speak about absolutely continuity of a measure $\mu$ with respect to $\m$ (see \cite[Definition 2.8, Proposition 2.9]{MRS21}) and that, in this case, a suitable extension of the Radon-Nikodym Theorem holds (see \cite[Theorem 2.10]{MRS21}). In fact, if a measure $\mu$ is absolutely continuous with respect to the quasi-Radon reference measure $\m$, then there exists a measurable function $f$ on $\X$ such that for any $B \in \mathcal{B}(\X)$ it holds 
\[\mu(B) = \int_B f \, \de \m.\]

In the following we will always refer to structures $(\X, \di, \m)$ where $(\X, \di)$ is a complete and separable metric space and $\m$ is a quasi-Radon measure on $\X$, $\m \neq 0$. Moreover, up to consider the metric measure space $(\supp(\m),\di,\m)$, we can assume, without losing generality, that $\m$ has full support, that is $\supp(\m)=\X$.

\subsection{Reduced Curvature-Dimension Condition $\CD^\ast(K, N)$ for $N < 0$}

Let $(\X, \di, \m)$ be a metric measure space and $N < 0$. We introduce the \emph{R\'enyi entropy} $S_{N,\m}$ with respect to the reference measure $\m$ as the functional defined on $\mathscr{P}(\X)$ by 
\begin{equation*}
S_{N,\m}(\mu):= \begin{cases}
\displaystyle \int_{\X} \rho(x)^{\frac{N-1}{N}} d\m(x) &\qquad \text{if } \, \mu \ll  \m, \, \mu = \rho \m,\smallskip\\
+\infty &\qquad \text{otherwise}.
\end{cases}
\end{equation*}
 As already pointed out in the introduction, if $\m$ is a Radon measure, the functional $S_{N,\m}$ is lower semicontinuous with respect to the weak topology and thus it is lower semicontinuous with respect to the Wasserstein convergence in $\Prob_2(\X)$. Unfortunately, the same property does not hold for quasi-Radon reference measures $\m$, but it is possible to prove the following nice result (see \cite[Proposition 4.8]{MRS21}):

\begin{prop}\label{prop:lsc}
In a metric measure space $(\X,\di,\m)$ the R\'enyi entropy functional $S_{N,\m}$ is weakly lower semicontinuous on the space
\begin{equation*}
    \Prob^\mathcal S(\X,\m):= \{\mu\in \Prob_2(\X) \,: \, \mu(\mathcal S_\m)=0\},
\end{equation*}
where $\mathcal S_\m$ is the singular set associated to the measure $\m$.
\end{prop}

Before going on, we define some useful subsets of $\Prob(\X)$ that will be relevant later on:
\begin{equation*}
\begin{split}
    \Prob_{ac}(\X,\m) &:= \{ \mu \in \ProbTwo(\X) : \mu \ll \m\},\\
    \Prob_\infty(\X,\m) &:=\{\mu \in \Prob_{ac}(\X,\m) : \text{$\mu$ has bounded support}\},\\
    \Prob^*_N (\X,\m) &:= \{ \mu \in \Prob_\infty(\X,\m ) : S_{N,\m}(\mu)<\infty\}.
\end{split}
\end{equation*}
Observe that, since $\m(\mathcal S_\m)=0$, it holds that $\Prob_{ac}(\X,\m)\subset \Prob^\mathcal S(\X,\m)$, thus in particular the entropy functional $S_{N,\m}$ is weakly lower semicontinuous on $\Prob_{ac}(\X,\m)$.

In order to give the definition of reduced curvature-dimension bounds for negative values of the dimensional parameter, we need also to introduce the following distortion coefficients for $N < 0$:
\begin{equation}\label{E:sigma}
\sigma_{K,N}^{(t)}(\theta):= 
\begin{cases}
\infty, & \textrm{if}\ K\theta^{2} \leq N\pi^{2}, \crcr
\displaystyle  \frac{\sin(t\theta\sqrt{K/N})}{\sin(\theta\sqrt{K/N})} & \textrm{if}\  N\pi^{2} < K\theta^{2} <  0, \crcr
t & \textrm{if}\ 
K \theta^{2}=0,  \crcr
\displaystyle   \frac{\sinh(t\theta\sqrt{-K/N})}{\sinh(\theta\sqrt{-K/N})} & \textrm{if}\ K\theta^{2} > 0
\end{cases}
\end{equation}
Notice that, for every $N<0$ and every $t\in[0,1]$, it holds that
\begin{equation}\label{eq:monotonicitysigma}
    \sigma_{K,N}^{(t)}(\theta) \text{ is increasing if $K\leq 0$ and decreasing if $K\geq0$.}
\end{equation}

\noindent We are then ready to define the notion of reduced curvature-dimension condition for negative values of the dimensional parameter $N$, which was introduced for the first time by Ohta in \cite{Ohta16}.

\begin{definition}
For any couple of measures $\mu_0,\mu_1 \in \mathscr{P}_{ac}(\X,\m)$, $\mu_i = \rho_i \m$ and any coupling $\pi\in \mathscr{P}(\X\times \X)$ between them, we denote by $R_{K,N}^{(t)}(\pi | \m)$  the functional defined by
\[
\begin{split}
R_{K,N}^{(t)}(\pi | \m)&:= \int_{\X\times \X} \Big[ \sigma^{(1-t)}_{K,N} \big(\di(x,y) \big) \rho_{0}(x)^{-\frac{1}{N}} +    \sigma^{(t)}_{K,N} \big(\di(x,y) \big) \rho_{1}(y)^{-\frac{1}{N}} \Big]   \de \pi(x,y).
\end{split}
\]
\end{definition}

\begin{definition}[$\CD^*$ condition]\label{def:curvcond}
For fixed  $K\in \R, N \in (-\infty,0)$, we say that a metric measure space $(\X, \di, \m)$ satisfies the  \emph{$\CD^*(K,N)$-condition}  if for each pair of measures $\mu_{0}=\rho_{0} \m,\, \mu_{1}=\rho_{1} \m \in \Prob_N^*(\X,\m)$  there exists an optimal coupling $\pi\in \OptPlans(\mu_0,\mu_1)$ and a $W_{2}$-geodesic  $\{\mu_t\}_{t \in [0, 1]} \subset \mathscr{P}_N^*(\X,\m)$ such that
\begin{align}\label{def:CD}
S_{N', \m}(\mu_t) \leq R_{K,N'}^{(t)}(\pi | \m)
\end{align}
holds for every $t\in [0,1]$ and every $N'\in [N, 0)$.
\end{definition}

\begin{remark}
We point out that Definition \ref{def:curvcond} is not exactly the same as the one adopted by Ohta in \cite{Ohta16}, where the condition is required to hold for every pair of absolutely continuous marginals. The main difference is basically that we ask the space $\Prob_N^*(\X,\m)$ to be geodesically convex even when inequality \eqref{def:CD} does not ensure it, as it may happen when $K<0$, due to the pathologies of the distortion coefficients. This is a technical detail but it is fundamental in this paper. On the other hand, let us point out that Definition \ref{def:curvcond} is consistent to the first definition of reduced $\CD$-condition, that was introduced by Bacher and Sturm in \cite{BacherSturm10} in the positive dimensional case.
\end{remark}

\noindent Finally we invite the reader to compare this reduced curvature-dimension condition $\CD^\ast(K, N)$ with the original $\CD(K, N)$ one (see \cite{Ohta16} and \cite{MRS21}).

\subsection{Essentially Non-Branching Metric Measure Spaces}

In this section we briefly introduce the essentially non-branching condition on a metric measure space, that was pioneered by Rajala and Sturm in \cite{RajalaSturm12}, proving also a result that will be fundamental in the following. This property is a weakening of the classical non-branching condition and it fits better into the context of metric measure spaces, since it takes into account also the reference measure. An example of these spaces is provided by Rajala and Sturm, who proved that every strong $\CD(K,\infty)$ space is essentially non-branching.

\begin{definition}
 In a metric space $(\X,\mathsf{d})$, a subset $G\subset \Geo(\X)$ is called non-branching if for any pair of geodesics $\gamma_1,\gamma_2\in G$ such that $\gamma_1\neq\gamma_2$, it holds that 
 \begin{equation*}
     \rest{0}{t} \gamma_1 \neq  \rest{0}{t} \gamma_1 \qquad \text{for every }t\in (0,1).
 \end{equation*}
 A metric measure space $(\X,\mathsf{d},\mathfrak{m})$ is said to be essentially non-branching if for every absolutely continuous measures $\mu_0,\mu_1\in\Prob_{ac}(\X,\m)$, every optimal geodesic plan $\eta$ connecting them is concentrated on a non-branching set of geodesics.
\end{definition}

The non-branching assumption for $(\X, \di)$ can be equivalently characterized by requiring that the map $(e_0, e_t) \colon \Geo(\X) \to \X^2$ is injective for some, and thus for any, $t \in (0, 1)$.

As one can intuitively realize, the essentially non-branching condition turns out to be useful in order to prevent mass overlap of two Wasserstein geodesics at any intermediate time. The following proposition provides an interesting result in this direction; similar statements can be find in (\cite[Lemma 2.8]{BacherSturm10} \cite[Lemma 3.11]{Erbar-Kuwada-Sturm13}), but, in our setting, we need a slightly more general result. The proof relies on a clever mixing argument, that was used for a different purpose in \cite{RajalaSturm12}. The same mixing procedure will be very useful in the proof of Proposition \ref{prop:ambrosiogigli}. 

\begin{prop}\label{prop:orthogonality}
Let $(\X, \di,\m)$ be an essentially non-branching metric measure space. Given the probability measures $\mu_0,\nu_0,\mu_1,\nu_1 \in \Prob_{ac}(\X,\m)$, consider two optimal geodesic plans $\ppi^\mu \in \OptGeo(\mu_0, \mu_1)$ and $\ppi^\nu \in \OptGeo(\nu_0, \nu_1)$, and their geodesic representation $(\mu_t)_{t \in [0,1]} = (e_t)_\# \ppi^\mu$ and $(\nu_t)_{t \in [0,1]}= (e_t)_\# \ppi^\nu$. Moreover, assume that $\ppi^\mu \perp \ppi^\nu$ and that there is an optimal transport plan $\pi \in \OptPlans$ such that
\[(e_0,e_1)_\# \ppi^\mu,(e_0,e_1)_\# \ppi^\nu \ll \pi.\]
Then for every $t \in (0,1)$ it holds $\mu_t \perp \nu_t$, provided that $\mu_t, \nu_t \ll \m$.
\end{prop}

\begin{pr}
The proof will be done by contradiction, thus let us assume that there exists $t\in (0,1)$ such that $\mu_t \not\perp \nu_t$. Consider the measures
\[
 \ppi^\text{left} = \frac12\left((\rest{0}{t})_\#\ppi^\mu + (\rest{0}{t})_\#\ppi^\nu\right), \qquad\ppi^\text{right} = \frac12\left((\rest{t}{1})_\#\ppi^\mu + (\rest{t}{1})_\#\ppi^\nu \right),
\]
then let $\{\ppi_x^\text{left}\}_{x\in X}$ be the disintegration of $\ppi^\text{left}$ with respect to $e_1$ and 
let $\{\ppi_x^\text{right}\}_{x\in X}$ be the disintegration of $\ppi^\text{right}$ with respect to $e_0$. Define also the measure 
\[\eta :=(e_0)_\# \ppi^\text{right} = (e_1)_\# \ppi^\text{left}= \frac{\mu_t+\nu_t}{2}= (e_t)_\# \bigg( \frac{\ppi^\mu+ \ppi^\nu}{2} \bigg).\]
Moreover let us consider the splitting map
\begin{align*}
    \text{Sp}:  C([0,1];\X) &\to \left\{(\gamma^1, \gamma^2) \in C([0,1];\X) \times C([0,1];\X) \,:\, \gamma_1^1 = \gamma_0^2\right\}\\
    \gamma &\mapsto (\rest{0}{t}\gamma, \rest{t}{1}\gamma),
\end{align*}
and notice that it is bi-Lipschitz, which in particular implies the existence of its (measurable) inverse 
\[\text{Sp}^{-1}:\left\{(\gamma^1, \gamma^2) \in C([0,1];\X) \times C([0,1];\X) \,:\, \gamma_1^1 = \gamma_0^2\right\} \to C([0,1];\X).\]
Now define the collection of measures $\{\ppi_x\}_{x \in \X} \subset \Prob(C([0,1];\X))$ as
\[
 \ppi_x : = (\text{Sp}^{-1})_\#(\ppi_x^\text{left} \times \ppi_x^\text{right})
\]
and finally introduce the ``mixed measure" $\ppi^\text{mix} \in \Prob(C([0,1];\X))$ as $$\ppi^\text{mix}(\dee \gamma)= \ppi_x(\dee \gamma)\eta(\dee x).$$
The next few passages of the proof are designed to prove that $\ppi^\text{mix}$ is an optimal geodesic plan. First of all notice that, since $\pi \in \OptPlans$, there exists a cyclically monotone set $\Gamma \subset \X \times \X$ with $\pi(\Gamma)=1$. Consider the set 
\begin{equation*}
    \tilde \Gamma = \{ \gamma \in \Geo(X) \suchthat (\gamma_0,\gamma_1) \in \Gamma\} = (e_0,e_1)^{-1} (\Gamma),
\end{equation*}
then, since $(e_0,e_1)_\# \ppi^\mu \ll \pi$, it holds 
\begin{equation*}
    \ppi^\mu (\tilde \Gamma)= \ppi^\mu \big((e_0,e_1)^{-1} (\Gamma)\big) = (e_0,e_1)_\# (\ppi^\mu) (\Gamma)=1,
\end{equation*}
and similarly $\ppi^\nu(\tilde \Gamma)=1$. Then for every pair of curves $\gamma_1,\gamma_2 \in \tilde\Gamma$ with $\gamma^1_t=\gamma^2_t$, the cyclical monotonicity, together with the triangular inequality, gives that
\begin{equation}\label{eq:chain}
\begin{split}
  \di^2(\gamma_0^1,\gamma_1^1) + \di^2(\gamma_0^2,\gamma_1^2)
 & \le \di^2(\gamma_0^1,\gamma_1^2) + \di^2(\gamma_0^2,\gamma_1^1)\\
 & \le \left(tl(\gamma^1) + (1-t)l(\gamma^2)\right)^2 + \left(tl(\gamma^2) + (1-t)l(\gamma^1)\right)^2\\
 & = l(\gamma^1)^2 + l(\gamma^2)^2 - 2t(1-t)\left(l(\gamma^1)-l(\gamma^2)\right)^2\\
 & \le l(\gamma^1)^2 + l(\gamma^2)^2 = \di^2(\gamma_0^1,\gamma_1^1) + \di^2(\gamma_0^2,\gamma_1^2),
\end{split}
\end{equation}
and so all the inequalities in the above chain \eqref{eq:chain} are equalities, in particular $l(\gamma^1) = l(\gamma^2)$. Thus, for every $x \in e_t(\tilde \Gamma)=:\tilde \Gamma_t$, there exists $l_x$ such that $l(\gamma)=l_x$, for every $\gamma \in \tilde \Gamma$ with $\gamma_t=x$. On the other hand, notice that 
\begin{equation*}
    \eta(\tilde \Gamma_t)= (e_t)_\# \bigg( \frac{\ppi^\mu+ \ppi^\nu}{2} \bigg) (e_t(\tilde \Gamma)) \geq \frac{\ppi^\mu+ \ppi^\nu}{2} (\tilde \Gamma)=1.
\end{equation*}
Moreover, using once again that $\ppi^\mu(\tilde \Gamma)=\ppi^\nu(\tilde \Gamma)=1$, follows that, for $\eta$-almost every $x\in X$, $\ppi_x^\text{left}$ is concentrated
on geodesics of type $\rest{0}{t} \gamma$ with $\gamma \in \tilde \Gamma$ and $\ppi_x^\text{right}$ is concentrated on geodesics of type $\rest{0}{t} \gamma$ with $\gamma \in \tilde \Gamma$. 
Therefore the measure $\ppi^\text{mix}$ is concentrated on the set
\[
\left\{\gamma \in C([0,1];X)\,:\, \text{there exist }\gamma^1,\gamma^2 \in \tilde \Gamma \text{ s.t. } \rest{0}{t}\gamma = \rest{0}{t}\gamma^1 \text{ and } \rest{t}{1}\gamma = \rest{t}{1}\gamma^2\right\},
\]
which, by the equalities in \eqref{eq:chain}, is a subset of $\Geo(X)$.\\
Furthermore, since $\ppi^\mu$, $\ppi^\nu$ and $\ppi^\text{mix}$ are concentrated on $\tilde \Gamma$, it holds that 
\begin{align*}
 \int_{\Geo(X)} \di^2(\gamma_0,\gamma_1)\de\bigg( \frac{\ppi^1+\ppi^2}2\bigg)(\gamma) 
  & = \int_{X} l_x^2\de \left[(e_t)_\#\bigg( \frac{\ppi^1+\ppi^2}2\bigg)\right](x)\\
  &= \int_X l_x^2 \de \eta(x)\\
  & = \int_{X} l_x^2\de((e_t)_\#\ppi^\text{mix})(x) =\int_{\Geo(X)}\di^2(\gamma_0,\gamma_1)\de\ppi^\text{mix}(\gamma).
\end{align*}
On the other hand $(e_0,e_1)_\#\left(\frac12(\ppi^\mu+\ppi^\nu)\right)$ is an optimal transport plan, while $(e_i)_\#\ppi^\text{mix} = (e_i)_\#\left(\frac12(\ppi^\mu+\ppi^\nu)\right)$ for $i = 0,1$, thus
the measure $\ppi^\text{mix}$ is an optimal geodesic plan.

Now, call $\rho^\mu_t, \rho^\nu_t$ the densities with respect to $\m$ of $\mu_t$ and $\nu_t$, respectively. 
Since $\mu_t \not\perp \nu_t$ there exists a set $E\subset X$ of positive $\m$-measure, where both $\rho^\mu_t$ and $\rho^\nu_t$ are strictly positive. Notice moreover that, for $\m$-almost every $x\in E$, at least one of the measures $\ppi_x^\text{left}$ and $\ppi_x^\text{right}$ is not a Dirac mass, because $\ppi^\mu \perp \ppi^\nu$. Therefore at least one between $\ppi^\text{mix}$ and its time-inverse $I_\#\ppi^\text{mix}$, defined via the mapping
\[
 I : \Geo(X) \to \Geo(X) \qquad \gamma \mapsto (\gamma' \colon [0,1] \to X \quad t \mapsto \gamma_{1-t}),
\]
is not concentrated in a non-branching set of geodesics. This contradicts the hypothesis of essentially non-branching.
\end{pr}

\section{Solution to Monge Problem}\label{section3}

In this section we investigate existence and/or uniqueness of an optimal transport map in essentially non-branching $\CD^*(K,N)$ spaces with $N < 0$. To this aim, we follow the approach proposed by Gigli in \cite{Gigli12a}, developing some technical results in order to apply it to our more general class of spaces. In particular, we start by stating and proving a uniqueness result when the marginals have finite entropy and bounded support, while at the end of the section we show an existence result for general marginals. 
For all these results the essentially non-branching assumption and  Proposition \ref{prop:orthogonality} play a central role, but we need also to add another preliminary proposition which is the extension of a result by Ambrosio and Gigli \cite[Proposition 2.16]{AmbrosioGigli11}. Notice that the proof of Proposition \ref{prop:ambrosiogigli} is similar to the one of Proposition \ref{prop:orthogonality}, as it relies on the same mixing argument.

\begin{prop} \label{prop:ambrosiogigli}
Let $(X,\di,\m)$ be an essentially non-branching metric measure space and let $(\mu_t)_{t\in[0,1]} \subset \ProbTwo(X)$ be a constant speed Wasserstein geodesic, such that $\mu_0,\mu_1\in \Prob_{ac}(\X,\m)$. Then for every $t \in (0,1)$ there exists only one optimal geodesic plan $\ppi^0 \in \OptGeo(\mu_0,\mu_t)$ and only one optimal geodesic plan $\ppi^1 \in \OptGeo(\mu_t,\mu_1)$ and they are both induced by a map from $\mu_t$.
\end{prop}

\begin{pr}
Pick $\ppi_0 \in \OptGeo(\mu_0,\mu_t)$ and $\ppi_1 \in \OptGeo(\mu_t,\mu_1)$. Hence let $\{ \ppi^0_x \}_{x\in X}$ be the disintegration of $\ppi^0$ with respect to $e_1$ and let $\{ \ppi^1_x \}_{x\in X}$ be the disintegration of $\ppi^1$ with respect to $e_0$. Furthermore (similarly to what was done in the proof of Proposition \ref{prop:orthogonality}), consider the splitting map 
\begin{align*}
    \text{Sp}:  C([0,1];\X) &\to \left\{(\gamma^1, \gamma^2) \in C([0,1];\X) \times C([0,1];\X) \,:\, \gamma_1^1 = \gamma_0^2\right\}\\
    \gamma &\mapsto (\rest{0}{t}\gamma, \rest{t}{1}\gamma).
\end{align*}
and its (measurable) inverse $$\text{Sp}^{-1}:\left\{(\gamma^1, \gamma^2) \in C([0,1];\X) \times C([0,1];\X) \,:\, \gamma_1^1 = \gamma_0^2\right\} \to C([0,1];\X).$$ Now define the collection of measures $\{\ppi_x\}_{x \in \X} \subset \Prob(C([0,1];\X))$ as
\[
 \ppi_x : = (\text{Sp}^{-1})_\#(\ppi_x^0 \times \ppi_x^1),
\]
and the measure $\ppi \in \Prob(C([0,1];X))$ as \[\ppi(\dee \gamma)= \ppi_x(\dee \gamma)\mu_t(\dee x).\]
In particular observe that, since $(e_1)_\# \ppi^0=(e_0)_\# \ppi^1=\mu_t$, it holds that
\begin{equation} \label{eq:restriction}
    (\res_0^t)_\# \ppi = \ppi^0 \qquad \text{and} \qquad (\res_t^1)_\# \ppi=\ppi^1.
\end{equation}
Set $\alpha:= (e_0,e_t,e_1)_\# \ppi \in \Prob(X^3)$. Now, using \eqref{eq:restriction}, it follows that 
\begin{equation}\label{eq:chain2}
\begin{split}
    \norm{\di(\p_1,\p_3)}_{L^2(\alpha)} &\leq \norm{\di (\p_1,\p_2) + \di (\p_2,\p_3)}_{L^2(\alpha)} \leq \norm{\di(\p_1,\p_2)}_{L^2(\alpha)} + \norm{\di(\p_2,\p_3)}_{L^2(\alpha)} \\
    & = \norm{\di(e_0,e_1)}_{L^2(\ppi^0)} + \norm{\di(e_0,e_1)}_{L^2(\ppi^1)}\\
    &= W_2(\mu_0,\mu_t) + W_2(\mu_t,\mu_1) = W_2(\mu_0,\mu_1),
\end{split}
\end{equation}
thus $(\p_1,\p_3)_\# \alpha=(e_0,e_1)_\# \ppi$ is an optimal transport plan. In particular, the first inequality in the chain \eqref{eq:chain2} must be an equality: this ensures that $\di(x,z)=\di(x,y)+\di(y,z)$ for $\alpha$-almost every $(x,y,z)$, which means that $x,y,z$ lie along a geodesic. Furthermore, since also the second inequality has to be an equality, the functions $(x,y,z)\mapsto\di(x,y)$ and $(x,y,z)\mapsto\di(y,z)$ are each a positive multiple of the other, for $\alpha$-almost every $(x,y,z)$. Thus, it follows that for $\alpha$-almost every $(x,y,z)$ it holds 
\begin{equation*}
    \di(x,y)= t \di (x,z) \qquad \text{and} \qquad \di(y,z)=(1-t)\di(x,z).
\end{equation*}
This last information, together with \eqref{eq:restriction}, ensures that $\ppi$ is concentrated on $\Geo(X)$. Moreover $(e_0,e_1)_\#\ppi$ is an optimal transport plan and so $\ppi$ is an optimal geodesic plan.

We are now going to prove that $\ppi^1$ is induced by a map, the proof for $\ppi^0$ is totally analogous. Notice that the essentially non-branching assumption guarantees that $\ppi^1_x$ is a Dirac mass for $\mu_t$-almost every $x\in X$, since otherwise $\ppi \in \OptGeo(\mu_0,\mu_1)$ will not be concentrated in a non-branching set of geodesic. In addition, the map $x \mapsto \ppi_x^1$ is Borel measurable, therefore there exists a Borel measurable map $T:X\to\Geo(X)$ such that $\ppi_x^1=\delta_{T(x)}$ for $\mu_t$-almost every $x\in X$. In particular $\ppi^1=T_\#\mu_t$ and this is sufficient to conclude the proof.
\end{pr}

\noindent Let us then present the following corollary which is an easy consequence of Proposition \ref{prop:ambrosiogigli}.

\begin{corollary}\label{lem:gigli1}
Let $(X,\di,\m)$ be an essentially non-branching $\CD^*(K,N)$ space, for some $K\in \R$ and $N<0$, and let $\ppi \in \OptGeo(\mu_0,\mu_1)$ be such that
\begin{equation} \label{eq:finiteness}
   \mu_t= \rho_t \m:= (e_t)_\# \ppi \in \Prob^*_N(\X,\m) \text{ for every } t\in[0,1].
\end{equation}
Then it holds that
\begin{equation*}
    \lim_{s\to 0} S_{N,\m}( \mu_s) = S_{N,\m} ( \mu_0).
\end{equation*}
\end{corollary}

\begin{pr}
Observe that Proposition \ref{prop:ambrosiogigli} guarantees that, for a fixed $r \in (0,1)$, there exists a unique optimal geodesic plan between $ \mu_0$ and $ \mu_r$, which is then $(\res_0^r)_\# \ppi$. Moreover notice that when $K<0$, since $\mu_0$ and $\mu_1$ have bounded support, we can find $r\in (0,1)$ such that 
\begin{equation*}
    \sup_{x\in \supp (\mu_0),\, y \in \supp(\mu_r)} \di (x,y) < \pi \sqrt{\frac N K }.
\end{equation*}
As a consequence of this uniqueness, \eqref{def:CD} must be true along $(\res_0^r)_\#  \ppi$, therefore for a suitable $q\in \mathsf{Opt}(\mu_0,\mu_r)$ it holds
\begin{equation*}
S_{N,\m}(\mu_{tr})\leq\int \Big[ \sigma^{(1-t)}_{K,N} \big(\di(x,y) \big) \rho_{0}(x)^{-\frac{1}{N}} +    \sigma^{(t)}_{K,N} \big(\di(x,y) \big) \rho_{r}(y)^{-\frac{1}{N}} \Big]   \de q (x,y).
\end{equation*}
Now, letting $t \to 0$, the following pointwise convergences holds
\begin{equation*}
    \sigma^{(1-t)}_{K,N} \big(\di(x,y) \big) \to 1  \quad \text{and} \quad \sigma^{(t)}_{K,N} \big(\di(x,y) \big)\to 0,
\end{equation*}
and applying the dominated convergence theorem we conclude that $\limsup_{s \to 0} S_{N,\m}( \mu_s) \leq  S_{N,\m}( \mu_0) $. On the other hand, since $\mu_0$ is in particular absolutely continuous with respect to $\m$, the lower semicontinuity of the entropy functional $S_{N,\m}$, proven in Proposition \ref{prop:lsc}, allows to deduce that $S_{N,\m}( \mu_s) \to S_{N,\m} ( \mu_0)$ as $s \to 0$.
\end{pr}

\noindent Notice that, refining the proof of Corollary \ref{lem:gigli1}, it is possible to prove that for an essentially non-branching $\CD^*(K,N)$ space $(\X,\di,\m)$, the entropy functional is continuous along every Wasserstein geodesic with domain in $\Prob^*_N(\X,\m)$. 
Another nice consequence of this last result is stated in the following lemma, which will represent a key element for the proof of the main theorem (Theorem \ref{thm:monge}).
\begin{lemma}\label{lem:gigli2}
Let $(X,\di,\m)$ be an essentially non-branching $\CD^*(K, N)$ space for some $K\in \R$ and $N<0$ and assume $\ppi \in \OptGeo(\mu_0,\mu_1)$ to be such that
\begin{equation} 
   \mu_t:= (e_t)_\# \ppi \in \Prob^*_N(\X,\m) \text{ for every } t\in[0,1].
\end{equation}
Assume also that there exists a set $B \subset X \setminus S$, with $\m(B)<\infty$, such that 
\begin{equation*}
    \supp(\mu_t)\subset B \qquad \text{for every $t\in [0,1]$}.
\end{equation*}
Call $\rho_t$ the density of $\mu_t$ with respect to the reference measure $\m$, that is $\mu_t=\rho_t\m$, then 
\begin{equation*}
    \m(\{\rho_0>0\}) \leq \liminf_{t\to 0} \m(\{\rho_t>0\}).
\end{equation*}
\end{lemma}

\begin{pr}
Observe that, since $\mu_0\ll\m$, the set $\{\rho_0>0\}$ has positive $\m$ measure. Moreover $\{\rho_0>0\}\subset B$ up to a $\m$-null set, thus $0<\m(\{\rho_0>0\})<\infty$. So fix $\epsilon>0$ sufficiently small and take a Borel set $A_\epsilon\subset \{\rho_0>0\}$ such that $\m(\{\rho_0>0\})-\m(A_\epsilon)<\epsilon$ and $c<\rho_0(x)< C$ for every $x \in A_\epsilon$ and suitable constants $c$ and $C$. Now define the set $\mathcal{A}_\epsilon:= (e_0)^{-1}(A_\epsilon)$ and consequently the measures 
$\ppi',\ppi'' \in \Prob(\Geo(\X))$ as
\begin{equation*}
    \ppi':= \frac{1}{\ppi(\mathcal A_\epsilon)} \restr{\ppi}{\mathcal{A}_\epsilon} \quad \text{and} \quad \dee \ppi''(\gamma) = \frac{1}{\m(A_\epsilon)\rho_0(\gamma_0)} \dee \ppi' (\gamma).
\end{equation*}
By construction $\ppi''\ll\ppi$ with bounded density, thus if follows that $\tilde \mu_t=\tilde\rho_t \m := (e_t)_\# \ppi''\in \Prob^*_N(\X)$ for every $t\in [0,1]$ and $ \ppi'' \in \OptGeo(\tilde \mu_0,\tilde \mu_1)$. Furthermore, it is easy to realize that $\tilde \mu_0= \m(A_\epsilon)^{-1} \restr{\m}{A_\epsilon}$ and $\m(\{\tilde\rho_t>0\})\leq \m(\{\rho_t>0\})$ for every $t\in[0,1]$. Once again $\{\tilde\rho_t>0\}\subset B$ up to a $\m$-null set, therefore $\m(\{\tilde\rho_t>0\})<\infty$ and consequently applying Jensen's inequality we can deduce that
\begin{align*}
    S_{N,\m}(\tilde\mu_t) &= \int \tilde \rho_t^{1-\frac1N} \de \m = \int_{\{\tilde\rho_t>0\}} \tilde \rho_t^{1-\frac1N} \de \m = \m(\{\tilde\rho_t>0\}) \int \tilde \rho_t^{1-\frac1N}  \de \bigg[ \frac{\restr{\m}{\{\tilde\rho_t>0\}}}{\m(\{\tilde\rho_t>0\})}\bigg] \\
    & \geq \m(\{\tilde\rho_t>0\}) \left(\int \tilde \rho_t  \de \bigg[ \frac{\restr{\m}{\{\tilde\rho_t>0\}}}{\m(\{\tilde\rho_t>0\})}\bigg] \right)^{1-\frac1N}=\m(\{\tilde\rho_t>0\})^\frac1N
\end{align*}
It is then possible to apply Corollary \ref{lem:gigli1} to $\ppi''$ and deduce that
\begin{equation*}
    \lim_{t \to 0} S_{N,\m}(\tilde \mu_t)= S_{N,\m} (\tilde \mu_0)=  (\m(A_\epsilon))^\frac{1}{N}.
\end{equation*}
Consequently it holds that 
\begin{equation*}
    \m(\{\rho_0>0\} ) - \epsilon< \m(A_\epsilon) \leq \liminf_{t\to 0} \m(\{\tilde\rho_t>0\} )\leq \liminf_{t\to 0} \m(\{\rho_t>0\} ).
\end{equation*}
The thesis follows from the arbitrariness of $\epsilon$.
\end{pr}

\noindent We are now ready to prove the main result of this section, which provides the uniqueness of Wasserstein geodesics in $\Prob^*_N(\X,\m)$. 

\begin{theorem}\label{thm:monge}
Let $(X,\di,\m)$ be an essentially non-branching $\CD^*(K,N)$ space, for some $K\in \R$ and $N<0$. Then, for every $\mu_0,\mu_1\in\Prob^*_N(\X,\m)$ there exists a unique $\ppi\in\OptGeo(\mu_0,\mu_1)$ which satisfies 
\begin{equation}\label{eq:KNconvexity}
 \mu_t:= (e_t)_\# \ppi \in \Prob_N^*(\X,\m) \text{ for every } t\in[0,1],
\end{equation} 
and it is induced by a map.
\end{theorem}

\begin{pr}
First of all notice that, since the functional $S_{N,\m}$ is convex with respect to linear interpolation, it is sufficient to prove that every $\ppi\in\OptGeo(\mu_0,\mu_1)$ satisfying \eqref{eq:KNconvexity} is induced by a map. So assume by contradiction that there are two measures $\mu_0= \rho_0 \m,\mu_1= \rho_1 \m \in \Prob^*_N(\X,\m)$ and a plan $\ppi\in\OptGeo(\mu_0,\mu_1)$ satisfying \eqref{eq:KNconvexity} which is not induced by a map. Call $\{\ppi_x\}_{x\in \X} \subset \Prob(\Geo(\X))$ the disintegration kernel of $\ppi$ with respect to the map $e_0$. Hence there exists a $\mu$-positive set $C$, such that $\rho_0$ is positive $\m$-almost everywhere on $C$ and $\ppi_x$ is not a delta measure for every $x\in C$. As a consequence, the measure 
\begin{equation*}
    \eta= \int_C \ppi_x \times \ppi_x \de \mu(x) 
\end{equation*}
is not concentrated in in the diagonal $D := \{(\gamma,\gamma): \gamma \in \Geo(\X)\}$. Therefore there exists a point $(\gamma_1,\gamma_2)\in \supp(\eta) \subset \Geo(\X)\times \Geo(\X)$ with $\gamma_1 \ne \gamma_2$. Take $\varepsilon>0$ small enough such that $B(\gamma_1,\varepsilon)\cap B(\gamma_2,\varepsilon)=\emptyset$, then 
\begin{equation*}
    \eta(B(\gamma_1,\varepsilon)\times B(\gamma_2,\varepsilon))> 0 
\end{equation*}
and consequently, up to restricting $C$, we can assume that 
\begin{equation}\label{eq:positivity}
    \ppi_x(B(\gamma_1,\varepsilon)),\ppi_x(B(\gamma_1,\varepsilon))>0 
\end{equation}
for every $x \in C$. On the other hand, if $\mathcal S_\m$ is the singular set for the measure $\m$ introduced in \eqref{def:Sm}, the fact that $\m(\mathcal S_\m)=0$ and that $\mu$ is absolutely continuous with respect to $\m$ ensure that $\mu(\mathcal S_\m)=0$. Thus, for every $\delta>0$ define the open set 
 \begin{equation*}
 S^\delta = \big\{ x \in X \suchthat \inf_{s\in \mathcal S_\m} \di (x,s) < \delta \big\}
 \end{equation*}
 and notice that, since $\mathcal S_\m$ is closed, 
 \begin{equation*}
     \bigcap_{\delta>0} S^\delta = \mathcal S_\m,
 \end{equation*}
 and consequently $\mu(S^\delta)\to0$ as $\delta \to 0$. Therefore for a suitably small $\tilde\delta>0$ it holds that 
 \begin{equation*}
     \mu\big(C \setminus S^{\bar\delta} \big) >0,
 \end{equation*}
 thus, up to further restrict $C$, we can assume that $C \subset \big(S^{\bar\delta}\big)^c$. Then, since $\mu$ is a Radon measure measure and therefore it is inner regular, we can also assume that $C$ is compact and consequently $\m(C)<\infty$. Moreover, the reference measure $\m$ is locally finite when restricted to $X\setminus \mathcal S_\m$, thus for every $\delta$ small enough $\m\big(C^{ \delta}\big)<\infty$, where
 \begin{equation*}
     C^{\delta} = \big\{ x \in X \suchthat \inf_{c\in C} \di (x, c) < \delta \big\}.
 \end{equation*}
 Furthermore, since $C$ is closed, there exists $\tilde \delta >0$ such that $\m\big(C^{\tilde \delta}\big)<\frac 32 \m (C)$. \\
 Now, introduce the sets $\Gamma_1,\Gamma_2\in \Geo(X)$ as
 \begin{equation*}
     \Gamma_1:= B(\gamma_1,\varepsilon) \cap e_0^{-1}(C) \quad \text{and} \quad \Gamma_2:= B(\gamma_2,\varepsilon) \cap e_0^{-1}(C),
 \end{equation*}
 and notice that \eqref{eq:positivity} ensures that $\ppi(\Gamma_1),\ppi(\Gamma_2)>0$. Consequently define $\ppi^1,\ppi^2\in \Prob(\Geo(\X))$ by posing
 \begin{equation*}
     \ppi^1 := \frac{\restr{\ppi}{\Gamma_1}}{\ppi(\Gamma_1)} \quad \text{and} \quad \ppi^2 := \frac{\restr{\ppi}{\Gamma_2}}{\ppi(\Gamma_2)}.
 \end{equation*}
 Observe that $\ppi^1,\ppi^2\ll \ppi$ with bounded density, thus $\mu^1_t:= (e_t)_\# \ppi^1 \in \Prob^*_N(\X,\m)$ and $\mu^2_t:= (e_t)_\# \ppi^2 \in \Prob^*_N(\X,\m)$ for every $t\in [0,1]$. Moreover, since $B(\gamma_1,\varepsilon)\cap B(\gamma_2,\varepsilon)=\emptyset$, it is clear that $\ppi^1\perp\ppi^2$, then, after noticing that $(e_0,e_1)_\# \ppi^1,(e_0,e_1)_\# \ppi^2 \ll (e_0,e_1)_\# \ppi\in \OptPlans$, it is possible to apply Proposition \ref{prop:orthogonality} and conclude that $\mu_t^1\perp\mu_t^2$ for every $t\in(0,1)$.
 On the other hand, the sets $e_1(B(\gamma_1,\varepsilon))$ and $e_1(B(\gamma_2,\varepsilon))$ are bounded, therefore there exists a time $\bar t$ such that $\supp(\mu^1_t), \supp(\mu_t^2) \subset C^{\tilde \delta}$ for every $t\in[0,\bar t]$. Hence we can apply Lemma \ref{lem:gigli2} to the measures $\rest{0}{\bar t}_\#\ppi^1$ and $\rest{0}{\bar t}_\#\ppi^2$ (with $B = C^{\tilde \delta}$) and deduce that 
 \begin{equation*}
   \m(C)= \m(\{\rho^1_0>0\}) \leq \liminf_{t\to 0} \m(\{\rho^1_t>0\}) \quad \text{and}  \quad \m(C)=\m(\{\rho^2_0>0\}) \leq \liminf_{t\to 0} \m(\{\rho^2_t>0\}),
\end{equation*}
where $\rho^1_t$ and $\rho^2_t$ denote the density of $\mu^1_t$ and $\mu^2_t$ with respect to the reference measure $\m$. In particular, for $t<\bar t$ sufficiently small, using that $\mu_t^1\perp\mu_t^2$ we can conclude that
\begin{equation*}
    \frac32 \m(C) < \m(\{\rho_t^1>0\})+\m(\{\rho_t^2>0\}) \leq \m(C^{\tilde \delta})\leq \frac32 \m(C), 
\end{equation*}
obtaining the desired contradiction.
\end{pr}

\noindent This uniqueness result can be used in order to show the existence of a transport map between two general absolutely continuous marginals, with possibly unbounded support. We point out that the existence of a transport map is a global property and in general it cannot be studied locally and then globalized. Anyway the subsequent proof needs to be done by approximation and this is possible only thanks to the uniqueness provided by Theorem \ref{thm:monge}. 

\begin{corollary}
Let $(\X,\di,\m)$ be an essentially non-branching $\CD^*(K,N)$ space, for some $K\in \R$ and $N<0$. Then, for every $\mu_0,\mu_1\in\ProbTwo(\X)$ which are absolutely continuous with respect to $\m$ there exists $\ppi\in\OptGeo(\mu_0,\mu_1)$ which is induced by a map.
\end{corollary}

\begin{pr}
First of all assume that $\mu_0\in \Prob_N^*(\X,\m)$ and fix $\ppi\in\OptGeo(\mu_0,\mu_1)$. Consider a countable and measurable family of disjoint bounded sets $(F_n)_{n\in \setN^+}$, covering $\mu_1$-almost all $\X$ (that is $\mu_1(\X\setminus\cup_{n}F_n)=0$), such that $\mu_1(F_n)>0$ for every $n$ and \begin{equation*}
    S_N\bigg(\frac{\restr{\mu_1}{F_n}}{\mu_1(F_n)}\bigg)<\infty \qquad \text{for every }n\in\setN.
\end{equation*}
We are going to define inductively a sequence $(\ppi_n)_{n\in \setN} \subset \Prob(\Geo(\X))$, with the following properties (where $E_n= \cup_{i=1}^n F_i$)
\begin{itemize}
    \item[(i)] for every $n \in \setN$, we have $\ppi_n\in \OptGeo(\mu_0,\mu_1)$ thus in particular $(e_0)_\#\ppi_n=\mu_0$ and $(e_1)_\#\ppi_n=\mu_1$,
    \item[(ii)] for every $n\in \setN^+$ and every $m \in \setN^+$ such that $m<n$, it holds that $\restr{\ppi_n}{e_1^{-1}( E_m)}=\restr{\ppi_m}{e_1^{-1}( E_m)}$,
    \item[(iii)] for every $n\in \setN^+$ the optimal geodesic plan $\frac{1}{\mu_1(E_n)}\restr{\ppi_n}{e_1^{-1}( E_n)}$ is induced by a map and constitutes a Wasserstein geodesic contained in $\Prob_N^*(\X)$.
\end{itemize}
 First of all put $\ppi_0=\ppi$, which obviously satisfies all the required properties. Then, given $\ppi_{n-1}$, call $\mu_0^n$ and $\mu_1^n$ the marginals at time $0$ and $1$ (respectively) of $\frac{1}{\mu_1(F_n)}\restr{\ppi_{n-1}}{e_1^{-1}( F_n)}$. Then consider the optimal geodesic plan $\tilde \ppi_n \in \OptGeo(\mu_0^n,\mu_1^n)$ provided by Theorem \ref{thm:monge}. Notice that $\mu_1^n\in \Prob_N^*(\X, \m)$ because of the definition of $F_n$ and $\mu_0^n\in \Prob_N^*(\X, \m)$ because $\mu_0^n \ll \mu_0$ with bounded density.
 Consequently define 
 \begin{equation}\label{eq:induction}
     \ppi_n= \ppi_{n-1} - \restr{\ppi_{n-1}}{e_1^{-1}( F_n)} + \mu_1(F_n) \tilde \ppi_n,
 \end{equation}
 it is easy to realize that $\ppi_n$ satisfies (i). Using the inductive assumption it is also clear that property (ii) holds, in fact from \eqref{eq:induction} follows that $\ppi_n$ and $\ppi_{n-1}$ coincide on $e_1^{-1}(F_n^c)$. Moreover it holds that
 \begin{equation*}
     \frac{1}{\mu_1(E_n)}\restr{\ppi_n}{e_1^{-1}( E_n)} = \frac{1}{\mu_1(E_{n})}\restr{\ppi_{n-1}}{e_1^{-1}( E_{n-1})} + \frac{\mu_1(F_n)}{\mu_1(E_{n})} \tilde \ppi_n
 \end{equation*}
 and therefore $\frac{1}{\mu_1(E_n)}\restr{\ppi_n}{e_1^{-1}( E_n)}$ constitutes a Wasserstein geodesic contained in $\Prob_N^*(\X, \m)$. As a consequence, Theorem \ref{thm:monge} ensures that it is induced by a map, proving (iii). Now the combination of properties (i), (ii) and (iii) implies that $(\ppi_n)_{n\in \setN}$ is a Cauchy sequence with respect to the total variation norm and thus it converges to $\tilde \ppi \in \OptGeo(\mu_0,\mu_1)$ with the property that $\restr{\tilde \ppi}{e_1^{-1}(E_n)}=\restr{ \ppi_n}{e_1^{-1}(E_n)}$ for every $n\in \setN^+$. We are now going to prove that $\ppi$ is induced by a map. Assume by contradiction this is not true. Calling $\{\tilde \ppi_x\}_{x\in X} \subset \Prob(\Geo(X))$ the disintegration of $\tilde \ppi$ with respect to the map $e_0$, $\tilde\ppi_x$ is not a delta measure for a $\mu$-positive set of $x$. Then, since $\mu_1(X\setminus\cup_{n}F_n)=0$, there exists $\bar n \in \setN$ such that $\restr{\tilde\ppi_x}{e_1^{-1}(E_{\bar n})}$ is not a delta measure for a $\mu$-positive set of $x$, contradicting the fact that $\restr{\tilde \ppi}{e_1^{-1}(E_{\bar n})}$ is induced by a map.\\
 We can now explain the proof of the general case, assuming only the absolute continuity on the first marginal. This proof can be done using an approximation procedure, very similar the one showed in the first part of the proof, for this reason we will not explain all the passages. As before, let us consider a countable and measurable family of disjoint bounded sets $(F_n)_{n\in \setN^+}$, covering $\mu_0$-almost all $X$, such that $\mu_0(F_n)>0$ for every $n$ and \begin{equation*}
    S_N\bigg(\frac{\restr{\mu_0}{F_n}}{\mu_0(F_n)}\bigg)<\infty \qquad \text{for every }n\in\setN.
\end{equation*}
Then it is possible to proceed as before, using an inductive procedure (and the first part of the proof) in order to obtain $\tilde \ppi \in \OptGeo(\mu_0,\mu_1)$ such that $\restr{\tilde \ppi}{e_0^{-1}(E_n)}$ (where $E_n= \cup_{i=1}^n F_n$) is induced by a map for every $n \in \setN^+$. Therefore, this is sufficient to conclude that $\tilde \ppi$ is induced by a map.
\end{pr}

\section{Local-to-global Property}

In this last section we prove a local-to-global result for the reduced $\CD$ condition with a negative dimensional parameter. For this purpose, in this section we assume the metric measure space $(\X, \di, \m)$ to be locally compact.

\begin{definition}
For fixed  $K\in \R, N \in (-\infty,0)$, we say that a metric measure space $(\X,\mathsf{d},\m)$ satisfies the condition $\mathsf{CD}^*(K-,N)$ if for every $\mu_{0}=\rho_{0} \m,\, \mu_{1}=\rho_{1} \m \in \Prob_N^*(\X,\m)$ and every $K'<K$ there exists an optimal coupling $\pi\in \OptPlans(\mu,\nu)$ and a $W_{2}$-geodesic  $\{\mu_t\}_{t \in [0, 1]} \subset \mathscr{P}_N^*(\X,\m)$ connecting $\mu_0$ and $\mu_1$ such that
\begin{align}\label{def:CD-}
S_{N', \m}(\mu_t) \leq R_{K',N'}^{(t)}(\pi | \m)
\end{align}
holds for every $t\in [0,1]$ and every $N'\in [N, 0)$.
\end{definition}

We point out that, unlike to what happens in the positive dimensional case, the $\mathsf{CD}^*(K-,N)$ condition is not equivalent to the $\mathsf{CD}^*(K,N)$ one in general. This is basically due to the pathologies of the distortion coefficients $\sigma_{K,N}^{(t)}$ when $K<0$, on the other hand, if the curvature parameter is non-negative, this equivalence holds. We underline that the proof in the positive dimensional case relies on the lower semicontinuity of the entropy functionals, which does not hold in our context. Anyway we can overcome this difficulty using the uniqueness results obtained in section \ref{section3}.

\begin{prop}\label{prop:positivecurv}
For fixed  $K\geq 0, N \in (-\infty,0)$, a metric measure space $(\X,\mathsf{d},\m)$ satisfies the condition $\mathsf{CD}^*(K,N)$ if and only if it satisfies the $\mathsf{CD}^*(K-,N)$ one.
\end{prop}

\begin{proof}
The ``only if'' part of the statement is obviously true. In order to prove the ``if'' part, we start by noticing that Theorem \ref{thm:monge} ensure that every pair of marginals $\mu_{0},\, \mu_{1} \in \Prob_N^*(\X,\m)$ is connected by a unique geodesic $(\mu_t)_{t\in [0,1]}$ with domain in $\Prob_N^*(\X,\m)$. In particular we can take a sequence $(K_n)_{n \in \setN}$ such that $K_n \nearrow K$, and find for every $n$ an optimal plan $\pi_n\in \OptPlans(\mu_0,\mu_1)$ such that
\begin{equation}\label{eq:Kn}
    S_{N,\m} (\mu_t)\leq R_{K_n,N}^{(t)}(\pi_n|\m), \quad \text{for every }t \in [0,1].
\end{equation}
The sequence $(\pi_n)_{n\in \setN}$ is obviously tight, since every measure has the same marginals, thus up to taking a suitable subsequence there exists $\pi \in \OptPlans(\mu_0,\mu_1)$ such that $\pi_n\weakto \pi$. In order to conclude the proof it is sufficient to show that, for every $t\in [0,1]$, 
\begin{equation}\label{eq:RKN}
    \lim_{n\to \infty} R_{K_n,N}^{(t)}(\pi_n|\m) = R_{K,N}^{(t)}(\pi|\m),
\end{equation}
in fact this will allow to pass \eqref{eq:Kn} at the limit as $n\to \infty$. To this aim we just need to prove that 
\begin{equation*}
   \lim_{n\to \infty} \int \sigma_{K_n,N}^{(1-t)} (\di(x,y)) \rho_0(x)^{-\frac 1N} \de \pi_n = \int \sigma_{K,N}^{(1-t)} (\di(x,y)) \rho_0(x)^{-\frac 1N} \de \pi ,
   \end{equation*}
   the other term of $R_{K,N}^{(t)}$ can be treated analogously. Notice that, since $S_{N, \m}(\mu_0)<\infty$, then $\rho_0(x)^{- 1/N}\in L^1(\mu_0)$ thus, according to \cite[Lemma 2.12]{MRS21},
for every fixed $\varepsilon>0$ there exists $f^\varepsilon\in C_b(\X)$ such that $\norm{\rho_0^{- 1/N}-f^\varepsilon}_{L^1(\mu_0)}<\varepsilon$. Moreover, notice that  the coefficients $\sigma_{K_n,N}^{(1-t)}$ and $\sigma_{K,N}^{(1-t)}$ are uniformly bounded above by $1$ and continuous.  Moreover, it is easy to realize that
\begin{equation*}
    C_b(\X \times \X) \ni \sigma_{K_n,N}^{(1-t)} (\di(x,y)) f^\varepsilon(x) \rightarrow \sigma_{K,N}^{(1-t)} (\di(x,y)) f^\varepsilon(x) \in C_b(\X \times \X)
\end{equation*}
uniformly. As a consequence, the weak convergence $(\pi_n)_n \rightharpoonup \pi$ ensures that,
\begin{equation*}
    \lim_{n\to \infty} \int \sigma_{K_n,N}^{(1-t)} (\di(x,y)) f^\varepsilon(x) \de \pi_n = \int_{X\times X} \sigma_{K,N}^{(1-t)} (\di(x,y)) f^\varepsilon(x) \de \pi.
\end{equation*}
Furthermore, the uniform bound on $\sigma_{K_n,N}^{(1-t)}$ and $\sigma_{K,N}^{(1-t)}$ allows to deduce the following estimate
\begin{equation*}
    \begin{split}
        \limsup_{n\to \infty} \int \sigma_{K_n,N}^{(1-t)} (\di(x,y)) \rho_0(x)^{-\frac 1N} \de \pi_n &\leq \lim_{n\to \infty} \int \sigma_{K_n,N}^{(1-t)} (\di(x,y)) f^\varepsilon(x) \de \pi_n + \varepsilon  \\
        &= \int \sigma_{K,N}^{(1-t)} (\di(x,y)) f^\varepsilon(x) \de \pi + \varepsilon \\
        &\leq \int \sigma_{K,N}^{(1-t)} (\di(x,y)) \rho_0(x)^{-\frac 1N} \de \pi + 2  \varepsilon  .
    \end{split}
\end{equation*}
Analogously,  it can be proven that 
\begin{equation*}
    \liminf_{n\to \infty} \int \sigma_{K_n,N}^{(1-t)} (\di(x,y)) \rho_0(x)^{-\frac 1N} \de \pi_n \geq \int \sigma_{K,N}^{(1-t)} (\di(x,y)) \rho_0(x)^{-\frac 1N} \de \pi - 2  \varepsilon ,
\end{equation*}
and since $\varepsilon>0$ and can be chosen arbitrarily, equation \eqref{eq:RKN} holds true.\\
\end{proof}

In order to prove the local-to-global property we need a preliminary proposition, which states an equivalent characterization of the $\CD^\ast(K-,N)$ condition. The analogous result for the $\CD^\ast(K,N)$ condition for positive $N$ is proven in \cite[Proposition 2.8]{BacherSturm10}, but the approximation argument used relies on the lower semicontinuity of the entropy functionals. For this reason, we need to work with the $\CD^\ast(K-,N)$ condition and to proceed in a different way. 

\begin{prop}[Equivalent characterizations] \label{prop:equivalent}
Given a proper essentially non-branching metric measure space $(\X,\mathsf{d},\m)$, $K \in \R$ and $N < 0$, the following statements are equivalent:
\begin{itemize}
\item[(i)]$(\X,\mathsf{d},\m)$ satisfies the condition $\mathsf{CD}^*(K-,N)$.
\item[(ii)]For every $K'<K$ and every pair of marginals $\mu_0,\mu_1\in\Prob_N^*(\X,\m)$ there exists a geodesic $\mu:[0,1]\rightarrow\Prob_N^*(\X,\m)$
connecting $\mu_0$ and $\mu_1$ such that, for all $t\in [0,1]$ and all $N'\in [N,0)$, it holds that
\begin{equation} \label{convt}
S_{N', \m}(\mu_t)\leq\sigma^{(1-t)}_{K',N'}(\theta)S_{N', \m}(\mu_0)+
\sigma^{(t)}_{K',N'}(\theta)S_{N', \m}(\mu_1),
\end{equation}
where
\begin{equation} \label{theta}
\theta:=
\begin{cases}
\inf_{x_0\in\mathcal{S}_0,x_1\in\mathcal{S}_1}\mathsf{d}(x_0,x_1),& \text{if $K\geq 0$},\\
\sup_{x_0\in\mathcal{S}_0,x_1\in\mathcal{S}_1}\mathsf{d}(x_0,x_1),& \text{if $K<0$},
\end{cases}
\end{equation}
denoting by $\mathcal{S}_0$ and $\mathcal{S}_1$ the supports of $\mu_0$ and $\mu_1$, respectively.
\end{itemize}
\end{prop}

\begin{proof}
(i) $\Rightarrow$ (ii): This implication easily follows from the monotonicity property of the coefficient $\sigma^{(t)}_{K,N'}$, see \eqref{eq:monotonicitysigma}.

(ii) $\Rightarrow$ (i): We prove this implication only for $K>0$, our argument applies without any major modification also when $K\leq 0$. Notice that it is sufficient to prove condition \eqref{def:CD-} for $0<K'<K$, because of the monotonicity properties of the coefficients $\sigma^{(t)}_{K,N}$. Thus we fix $0<K'<K$ and two measures $\mu_0,\mu_1\in\Prob_N^*(\X,\m)$, which in particular means that there exist $o\in X$ and $R>0$ such that $\supp(\mu_0),\supp(\mu_1)\subseteq B_R(o)$. We also fix $K'<\tilde K<K$ and consider an arbitrary coupling $\mathsf{\tilde q}\in \mathsf{Opt}(\mu_0,\mu_1)$. Now for every $n\in \setN$ let $\mathcal{C}^n=\{C_1^n,\dots , C^n_{m_n}\}$ be a (finite) Borel partition of $B_R(o)$, that is 
\begin{equation*}
    \bigcup_{i=1}^{m_n} C^n_i= B_R(o) \quad \text{ and } \quad C^n_i\cap C^n_j= \emptyset \text{ for every }i\neq j,
\end{equation*}
such that $\diam(C^n_i)\leq \frac{1}{2^{n+1}}$ for every $i=1,\dots,m_n$. Moreover we assume that, for every $n\in \setN$, $\mathcal{C}_{n+1}$ is consistent with $\mathcal{C}_{n}$, meaning that for every $i$ we have that $C^n_i= C^{n+1}_{i_1}\cup \cdots \cup C^{n+1}_{i_k}$ for a suitable choice of indices $i_1,\dots i_k$. Furthermore, for every $n$ we take 
\begin{equation*}
    \delta_n:= \bigg[2^n\bigg(1- \sqrt{\frac{K'}{\tilde K}}\bigg)\bigg]^{-1},
\end{equation*}
and we define 
\begin{equation*}
    \tilde I_n:=\left\{(i,j)\in \{1,\dots,m_n\}^2 \, :\, \inf_{x\in C_i^n,y\in C^n_j} \di (x,y)> \delta_n \right\}.
\end{equation*}
Hence, let us set
\begin{equation*}
    E_n:= \bigcup_{(i,j)\in \tilde I_n} C_i^n\times C_j^n \subset B_R(o)\times B_R(o),
\end{equation*}
and notice that $E_{n}\subseteq E_{n+1}$.
Consequently we introduce for every $n\in \setN$ the set
\begin{equation*}
    \bar I_n = \left\{(i,j)\in \tilde I_n \, :\,  (C_i^n\times  C^n_j ) \cap E_{n-1}=\emptyset \right\},
\end{equation*}
where we assume $E_{-1}=\emptyset$. Given these definitions it is easy to realize that, calling $D:=\{(x,x)\,:\,x\in \X\}$ the diagonal of $\X$, it holds
\begin{equation}\label{eq:Ibar}
    \bigcup_{n=1}^\infty \bigcup_{(i,j)\in \bar I_n} C_i^n\times C_j^n = B_R(o) \times B_R(o) \setminus D ,
\end{equation}
moreover $(C_i^n\times C_j^n) \cap (C_k^{m}\times C_l^{m})= \emptyset$ for every $n,m\in\setN$ and  $(i,j)\in I_n, \, (k,l)\in I_m$.
 Now, for every $n\in\setN$, we define the set of indices
 \begin{equation*}
     I_n := \left\{ (i,j)\in \bar I_n \, :\, \mathsf{\tilde q}(C_i^n\times C_j^n)>0  \right\}
 \end{equation*}
 and for every $(i,j)\in I_n$ the associated probability measures $\mu^{n,ij}_0$ and $\mu^{n,ij}_1$ by
\[\mu^{n,ij}_0(A):=\frac{1}{\alpha^n_{ij}}\mathsf{\tilde q}((A\cap C_i^n)\times C^n_j) \quad \text{and} \quad \mu^{n,ij}_1(A):=\frac{1}{\alpha_{ij}} \mathsf{\tilde q}(C^n_i\times(A\cap C^n_j)),\]
where each $\alpha^n_{ij}:=\mathsf{\tilde q}(C_i\times C_j)\not=0$ is the suitable normalization constant. Then we call
\[\mathcal S ^{n,ij}_0:=\mathsf{supp}(\mu^{n,ij}_0)\subseteq\overline{C^n_i}  \quad \text{and} \quad \mathcal S ^{n,ij}_1=: \mathsf{supp}(\mu^{n,ij}_1)\subseteq\overline{C^n_j}\] and, accordingly to \eqref{theta}, we introduce
\begin{equation*}
\theta^{n,ij}:=\inf_{x_0\in\mathcal{S}^{ij}_0,x_1\in\mathcal{S}^{ij}_1}\mathsf{d}(x_0,x_1).
\end{equation*}
Observe that, since $\text{diam}(C^n_i)\leq \frac{1}{2^{n+1}}$ for every $i=1,\dots,m_n$, it holds that $\di(x,y)-\frac{1}{2^{n}}\leq \theta^{ij}$ for every $x\in\mathcal S ^{ij}_0$ and $y\in\mathcal S ^{ij}_1$.
By the assumption (ii), for every $n\in \setN$ and $(i,j)\in I_n$ there exist $\mathsf q^{n,ij} \in \mathsf{Opt}(\mu^{n,ij}_0,\mu_1^{n,ij})$ and a Wasserstein geodesic $\mu^{n,ij}:[0,1]\rightarrow\Prob_\infty(\X,\di,\m)$, connecting $\mu^{n,ij}_0=\rho^{n,ij}_0\m$ to $\mu^{n,ij}_1=\rho^{n,ij}_1\m$ and satisfying
\begin{equation*}
    S_{N',\m}(\mu^{n,ij}_t)\leq\sigma^{(1-t)}_{\tilde K,N'}(\theta^{n,ij})S_{N',\m}(\mu^{n,ij}_0)+
\sigma^{(t)}_{\tilde K,N'}(\theta^{n,ij})S_{N',\m}(\mu^{n,ij}_1).
\end{equation*}
Then, since $\di(x,y)-\frac{1}{2^n}\leq \theta^{n,ij}$ for every $x\in\mathcal S ^{n,ij}_0$ and $y\in\mathcal S ^{n,ij}_1$, it holds that
\begin{equation}\label{eq:convij}
    \begin{split}
        S&_{N',\m}(\mu^{n,ij}_t)\\
        &\leq\int \Big[\sigma^{(1-t)}_{\tilde K,N'}(\mathsf{d}(x_0,x_1)-2^{-n})\rho^{n,ij}_0(x_0)^{-1/N'}+\sigma^{(t)}_{\tilde K,N'}(\mathsf{d}(x_0,x_1)-2^{-n})\rho^{n,ij}_1(x_1)^{-1/N'}\Big]\de\mathsf{q}^{n,ij}(x_0,x_1)\\
&\leq \int \Big[\sigma^{(1-t)}_{ K',N'}(\mathsf{d}(x_0,x_1))\rho^{n,ij}_0(x_0)^{-1/N'}+\sigma^{(t)}_{ K',N'}(\mathsf{d}(x_0,x_1))\rho^{n,ij}_1(x_1)^{-1/N'}\Big]\de\mathsf{q}^{n,ij}(x_0,x_1)\\
&= R_{ K',N'}^{(t)}(\mathsf q ^{n,ij}|\m)
    \end{split}
\end{equation}
for every $t\in[0,1]$ and every $N'\in[ N,0)$, where the second inequality is a consequence of 
\begin{equation*}
   \sqrt{\tilde K} \cdot (\di(x,y)-2^{-n}) \geq \sqrt{K'} \cdot \di(x,y),
\end{equation*}
which holds for $\mathsf q ^{n,ij}$-almost everywhere because of the definition of $\delta_n$ and $\tilde I_n$. On the other hand, if $\alpha^D:=\mathsf{\tilde q}(D)>0$ we call
$$\rho^D_0 \m = \mu^D_0:=\frac{1}{\alpha^D}(\p_1)_\#[\mathsf{\tilde q}|_D]= \frac{1}{\alpha^D} (\p_2)_\# [\mathsf{\tilde q}|_ D]=:\mu^{D}_1 = \rho_1^D \m,$$
then, putting $\mu^D_t\equiv \mu_0^D=\mu_1^D$ and $\mathsf q^D= \frac{1}{\alpha^D} \mathsf{\tilde q}|_{D}$, it obviously holds that
\begin{equation*}
    S_{N',\m}(\mu^{D}_t)\leq R_{ K',N'}^{(t)}(\mathsf q^D |\m),
\end{equation*}
for every $t\in[0,1]$ and every $N'\in[ N,0)$. As a consequence of \eqref{eq:Ibar} it is also possible to conclude that 
\begin{equation*}
    \mu_\iota:=\alpha^D\mu_\iota^D + \sum^\infty_{n=1} \sum_{(i,j)\in I_n} \alpha^n_{ij}\mu^{n,ij}_\iota \quad \text{ for }\iota=0,1.
\end{equation*}
We can now define 
$$\mathsf{q}:= \alpha^D\mathsf{q}^D + \sum^\infty_{n=1} \sum_{(i,j)\in I_n} \alpha^n_{ij}\mathsf{q}^{n,ij} \quad \text{and} \quad \mu_t:=\alpha^D\mu_t^D + \sum^\infty_{n=1} \sum_{(i,j)\in I_n} \alpha^n_{ij}\mu^{n,ij}_t \, \, \text{ for every }t\in [0,1],$$ where both the series converge in the total variation norm.
Then clearly $\mathsf{q}$ is an optimal coupling of $\mu_0$ and $\mu_1$ while $\mu_t$ defines a geodesic connecting them. It is also easy to realize that we can apply Proposition \ref{prop:orthogonality} and deduce that for every $t\in[0,1]$
\begin{equation*}
    \mu_t^{n,ij} \perp \mu_t^{m,kl} \, \,\text{ if }n\neq m \text{ or } n=m \text{ and } (i,j)\ne(k,l),
\end{equation*}
moreover
\begin{equation*}
    \mu_t^{n,ij} \perp \mu_t^D \,\,\text{ for every }n\in \setN \text{ and }(i,j)\in I_n.
\end{equation*}
As a consequence, for every $t\in[0,1]$ and every $N'\in [N,0)$ it holds that 
\[S_{N',\m}(\mu_t)=[\alpha^D]^{1-1/N'}S_{N',\m}(\mu^D_t)+\sum^\infty_{n=1} \sum_{(i,j)\in I_n}[\alpha^n_{ij}]^{1-1/N'}S_{N',\m}(\mu^{n,ij}_t).\]
On the other hand, keeping in mind the definitions of $\mathsf q$, $\alpha^{n,ij}$ and $\mu_\iota^{n,ij}$ for $\iota=0,1$, we can easily conclude that 
\[
\begin{split}
    [\alpha&^D]^{1-1/N'}R_{ K',N'}^{(t)} (\mathsf q^D|\m)+ \sum^\infty_{n=1} \sum_{(i,j)\in I_n}[\alpha_{ij}^n]^{1-1/N'} R_{ K',N'}^{(t)}(\mathsf q ^{n,ij}|\m) \\
    &=  \int \sigma^{(1-t)}_{ K',N'}(\mathsf{d}(x_0,x_1))[\alpha^D\rho^D_0(x_0)]^{-1/N'}+\sigma^{(t)}_{ K',N'}(\mathsf{d}(x_0,x_1))[\alpha^D\rho^D_1(x_1)]^{-1/N'}\Big]\de [\alpha^D\mathsf{q}^D](x_0,x_1)\\
   &\qquad+ \sum^\infty_{n=1} \sum_{(i,j)\in I_n}\int \Big[\sigma^{(1-t)}_{ K',N'}(\mathsf{d}(x_0,x_1))[\alpha^n_{ij}\rho^{n,ij}_0(x_0)]^{-1/N'}+\sigma^{(t)}_{ K',N'}(\mathsf{d}(x_0,x_1))[\alpha^n_{ij}\rho^{n,ij}_1(x_1)]^{-1/N'}\Big]\\&\hspace{13cm}\de [\alpha^n_{ij}\mathsf{q}^{n,ij}](x_0,x_1)\\
    &\leq \int \Big[\sigma^{(1-t)}_{ K',N'}(\mathsf{d}(x_0,x_1))\rho_0(x_0)^{-1/N'}+\sigma^{(t)}_{ K',N'}(\mathsf{d}(x_0,x_1))\rho_1(x_1)^{-1/N'}\Big]\de\mathsf{q}(x_0,x_1)= R_{ K',N'}^{(t)}(\mathsf q|\m)
\end{split}
\]
Combining these last two relations with \eqref{eq:convij}, we obtain that
\begin{equation*}
    S_{N',\m}(\mu_t) \leq R_{ K',N'}^{(t)}(\mathsf q|\m),
\end{equation*}
concluding the proof.
\end{proof}

We are now ready to state the main result of this section:

\begin{theorem}\label{thm:globloc}
Let $K, N \in \R$ with $N < 0$ and let $(\X,\mathsf{d},\m)$ be a locally compact, essentially non-branching metric measure space such that $\Prob_N^* (\X,\m)$ is a geodesic space. If $(\X,\mathsf{d},\m)$ satisfies the condition $\CD^\ast(K-, N)$ locally, then it satisfies the condition $\CD^\ast(K-, N)$ globally.
\end{theorem}

The proof of this theorem relies on Proposition \ref{prop:equivalent}, so our goal is to demonstrate $(ii)$; to this aim we fix $K'<K$.
Before presenting the proof, we introduce the basic construction which allows to show its validity.\\

Let us fix a metric measure space $(\X,\di,\m)$ as in the hypothesis of Theorem \ref{thm:globloc}, satisfying $\CD^\ast(K-, N)$  locally for some $K \in \R$ and $N < 0$. Observe that, since $\Prob_N^* (\X,\m)$ is a geodesic space, then $(\X,\di)$ is a length space. Therefore the metric version of the Hopf-Rinow theorem (see for example \cite[Theorem 2.4]{Ballmann}) ensures that $(\X,\di)$ is proper, being locally compact. 
Finally, we fix a metric ball $B_R(o)$ and for $k \in \N \cup \{0\}$, we introduce the following property, that we denote by $\mathsf C(k)$. We remark that this property is similar in spirit to the one proposed in \cite{BacherSturm10}, where it is formulated in terms of midpoints of geodesics. However, in our situation it is more straightforward to consider directly the whole geodesic, thanks to the uniqueness result proven in the previous section.\\

\noindent$\mathsf C(k)$: For each geodesic $\Gamma \colon [0, 1] \to \Prob_N^* (\X,  \m)$ such that $\supp(\mu_0),\supp(\mu_1)\subseteq B_R(o )$ and for each pair of times $s, t \in [0,1]$, such that $t-s = 2^{-k}$ the (restricted and reparameterized) geodesic $\Gamma$ between $ \Gamma(s)$ and $\Gamma(t)$, satisfies the inequality
\[
S_{N',\m}\big(\Gamma(s + r(t-s))\big) \le \sigma^{(1-r)}_{K', N'}( \theta^k) S_{N',\m}(\Gamma(s)) + \sigma^{(r)}_{K', N'}( \theta^k) S_{N',\m}(\Gamma(t)),
\]
for all $r \in [0, 1]$ and $N' \in [ N, 0)$, where
\[
\theta^0:= \inf_{\gamma \in \supp(\Gamma)} \mathsf d(\gamma(0), \gamma(1)) \quad \textrm{ and } \quad \theta^k := \frac{\theta^0}{2^k}.
\]
In the following we treat the case $K > 0$, the general one follows by analogous computations.

\begin{lemma}\label{lemma:k-1}
If $\mathsf C(k)$ is satisfied for some $k \in \N$, then also $\mathsf C(k-1)$ holds true.
\end{lemma}

\begin{proof}

Let $k \in \N$ be such that $\mathsf C(k)$ is satisfied and let $\Gamma$ be a geodesic in $\Prob_N^* (\X, \m)$ such that $\supp(\mu_0),\supp(\mu_1)\subseteq B_R(o )$, moreover we fix $s, t \in [0,1]$
with $t-s=2^{1-k}$. First of all, let us observe that property $\mathsf C(k)$ ensures that it holds
\begin{equation}\label{eq:Ck1}
\begin{split}
S_{N', \m}\Big( \Gamma\big(s+r\cdot 2^{-k}\big)  \Big) \le  \sigma^{(1-r)}_{K', N'}\big( \theta^{k}  \big) S_{N', \m}\big(\Gamma(s)\big) + \sigma^{(r)}_{K', N'}\big( \theta^{k}  \big)S_{N', \m}\Big( \Gamma\big(s+2^{-k}\big) \Big),
\end{split}
\end{equation}
as well as
\begin{equation}\label{eq:Ck2}
\begin{split}
S_{N', \m}\Big( \Gamma\big(s+2^{-k} +r\cdot 2^{-k}\big)  \Big) \le  \sigma^{(1-r)}_{K', N'}\big( \theta^{k} \big) S_{N', \m}\Big(\Gamma\big(s+2^{-k}\big)\Big) + \sigma^{(r)}_{K', N'}\big( \theta^{k}  \big)S_{N', \m}\big( \Gamma(t) \big),
\end{split}
\end{equation}
for all $N' \in [N, 0)$. Now, applying property $\mathsf C(k)$ between the times $s + 2^{-(k +1)}$ and $s + 3 \cdot 2^{-(k+1)}$, which are at distance $2^{-k}$, we can also obtain the following chain of inequalities
\[
\begin{split}
S_{N', \m}\Big(\Gamma\big(s+2^{-k}\big)\Big) &\le \sigma^{(1/2)}_{K', N'}(\theta^k) S_{N', \m}\Big(\Gamma\big(s+2^{-(k+1)}\big)\Big) + \sigma^{(1/2)}_{K', N'}(\theta^k) S_{N', \m}\Big(\Gamma\big(s+3\cdot 2^{-(k+1)}\big)\Big)\\
& \le \sigma^{(1/2)}_{K', N'}(\theta^k) \Big[ \sigma^{(1/2)}_{K', N'}(\theta^k)S_{N', \m}\big(\Gamma(s)\big)   + \sigma^{(1/2)}_{K', N'}(\theta^k) S_{N', \m}\Big(\Gamma\big(s+2^{-k}\big)\Big)  \Big]\\
&\quad + \sigma^{(1/2)}_{K', N'}(\theta^k) \Big[\sigma^{(1/2)}_{K', N'}(\theta^k) S_{N', \m}\Big(\Gamma\big(s+2^{-k}\big)\Big)   + \sigma^{(1/2)}_{K', N'}(\theta^k)S_{N', \m}\big(\Gamma(t)\big) \Big],
\end{split}
\]
that in particular leads to
\begin{equation}\label{eq:midpoint}
S_{N', \m}\Big(\Gamma\big(s+2^{-k}\big)\Big) \le \dfrac{\big(\sigma^{(1/2)}_{K', N'}(\theta^k) \big)^2}{1-2\big(\sigma^{(1/2)}_{K', N'}(\theta^k) \big)^2} \big[ S_{N', \m} \big(\Gamma(s)\big) + S_{N', \m}\big( \Gamma(t)\big)  \big].
\end{equation}
Hence, let us observe that
\[
\begin{split}
 \dfrac{\Big(\sigma^{(1/2)}_{K', N'}(\theta^k) \Big)^2}{1-2\Big(\sigma^{(1/2)}_{K', N'}(\theta^k) \Big)^2} &= \dfrac{\sinh^2\Big(\frac{ \theta^k}{2} \sqrt{-K'/N'}\Big)}{\sinh^2\big( \theta^k \sqrt{-K'/N'}\big)} \cdot \dfrac{\sinh^2\big( \theta^k \sqrt{-K'/N'}\big)} {{\sinh^2\big( \theta^k \sqrt{-K'/N'}\big)} - 2 \sinh^2\Big(\frac { \theta^k}{2} \sqrt{-K'/N'}\Big)}\\
 &=\dfrac{\sinh^2\Big(\frac{ \theta^k}{2} \sqrt{-K'/N'}\Big)}{\cosh^2\big( \theta^k \sqrt{-K'/N'}\big) - \cosh\big( \theta^k \sqrt{-K'/N'}\big) }\\
 &=\dfrac 12 \dfrac{1}{\cosh \big( \theta^k \sqrt{-K'/N'}\big)} = \sigma^{(1/2)}_{K', N'}(2 \theta^k).
 \end{split}
\]
Moreover, since $\theta^{k} = \frac 12 \theta^{k-1}$, it holds that $\sigma^{(1/2)}_{K', N'}(2 \theta^k) = \sigma^{(1/2)}_{K', N'}(\theta^{k-1})$ and we can rewrite inequality \eqref{eq:midpoint} as
\begin{equation}\label{eq:midpoint2}
S_{N', \m}\big(\Gamma\big(s+2^{-k}\big)\big) \le \sigma^{(1/2)}_{K', N'}(\theta^{k-1})  S_{N', \m} \big( \Gamma(s)\big) + \sigma^{(1/2)}_{K', N'}(\theta^{k-1})  S_{N', \m} \big(\Gamma(t)\big).
\end{equation}
Thus, let us consider the geodesic $\Gamma$ restricted (and reparametrized) between the times $s$ and $t$ by considering the curve $[0,1]\ni r \to \Gamma(s + r \cdot 2^{1-k})$. If $r \in [0, 1/2]$, inequality \eqref{eq:Ck1} ensures that
\[
\begin{split}
S_{N', \m}&\big(\Gamma(s + r \cdot 2^{1-k})\big) \le \sigma^{(1-2r)}_{K', N'}(\theta^k)S_{N', \m}\big(\Gamma(s)\big) + \sigma^{(2r)}_{K', N'}(\theta^k)S_{N', \m}\big(\Gamma(s + 2^{-k})\big)\\
&\overset{\eqref{eq:midpoint2}}{\le} \Big( \sigma^{(1-2r)}_{K', N'}(\theta^k) + \sigma^{(2r)}_{K', N'}(\theta^k) \sigma^{(1/2)}_{K', N}(\theta^{k-1})  \Big) S_{N', \m}\big(\Gamma(s)\big) + \sigma^{(2r)}_{K', N'}(\theta^k) \sigma^{(1/2)}_{K', N}(\theta^{k-1})   S_{N', \m}\big(\Gamma(t)\big).
\end{split}
\]
Now, a direct computation shows that
\begin{equation}\label{eq:coef1}
    \sigma_{K',N'}^{\big((x_r-s)/2^{-k}\big)}(\theta^k)\,\sigma_{K',N'}^{(1/2)}(\theta^{k-1}) = \sigma_{K',N'}^{\big((x_r-s)/{2^{-(k+1)}}\big)}(\theta^{k-1}).
\end{equation}
while, using the sum-to-product trigonometric formulas, it is also possible to prove that 
\begin{equation}\label{eq:coef2}
    \sigma_{K',N'}^{\big((s+2^{-k}-x_r)/{2^{-k}}\big)}(\theta^k)+ \sigma_{K',N'}^{\big((x_r - s)/{2^{-k}}\big)}(\theta^k)\,\sigma_{K',N'}^{(1/2)}(\theta^{k-1}) = \sigma_{K',N'}^{\big((s+2^{-(k+1)}-x_r)/2^{-(k+1)}\big)}(\theta^{k-1}),
\end{equation}
where $x_r$ is any time between $s$ and $s+2^{-(k+1)}$. Making use of these expressions, we can then write the bound on $S_{N', \m}\big(\Gamma(s + r \cdot 2^{1-k})\big)$ as
\[
S_{N', \m}\big(\Gamma(s + r \cdot 2^{1-k})\big) \le \sigma^{(1-r)}_{K', N'}(\theta^{k-1}) S_{N', \m}\big(\Gamma(s)\big) +  \sigma^{(r)}_{K', N'}(\theta^{k-1}) S_{N', \m}\big(\Gamma(t)\big),
\]
when $r \in [0, 1/2]$. In the case in which $r \in [1/2, 1]$, we can apply \eqref{eq:Ck2} to obtain that
\[
\begin{split}
S&_{N', \m}\big(\Gamma(s + r \cdot 2^{1-k})\big) \le \sigma^{(2-2r)}_{K, N'}(\theta^k)S_{N', \m}\big(\Gamma(s + 2^{-k})\big) + \sigma^{(2r-1)}_{K', N'}(\theta^k)S_{N', \m}\big(\Gamma(t)\big)\\
&\overset{\eqref{eq:midpoint2}}{\le}  \sigma^{(2-2r)}_{K', N'}(\theta^k) \sigma^{(1/2)}_{K', N'}(\theta^{k-1}) S_{N', \m}\big(\Gamma(s)\big)  + \Big(\sigma^{(2-2r)}_{K', N'}(\theta^k) \sigma^{(1/2)}_{K', N'}(\theta^{k-1})  + \sigma^{(2r-1)}_{K', N'}(\theta^k) \Big)   S_{N', \m}\big(\Gamma(t)\big).
\end{split}
\]
Using again the identities in \eqref{eq:coef1} and \eqref{eq:coef2}, we get the bound
\[
S_{N', \m}\big(\Gamma(s + r \cdot 2^{1-k})\big) \le \sigma^{(1-r)}_{K', N'}(\theta^{k-1}) S_{N', \m}\big(\Gamma(s)\big) +  \sigma^{(r)}_{K', N'}(\theta^{k-1}) S_{N', \m}\big(\Gamma(t)\big),
\]
also when $r \in [1/2, 1]$, which shows the validity of property $\mathsf C(k-1)$.
\end{proof}

Notice that if $\Gamma$ is a geodesic in $\Prob_N^* (\X, \m)$ such that $\supp(\mu_0),\supp(\mu_1)\subseteq B_R(o )$, then $\supp (\Gamma(t))\subseteq \bar B_{2R}(o )$ for every $t\in [0,1]$. The compactness of $\bar B_{2R}(o )$ implies the existence of a constant $\lambda > 0$, of finitely many disjoint sets $L_1, \dots, L_n$ covering $\bar B_{2R}(o )$ and closed sets $X_1, \dots , X_n$ with $B_\lambda(L_j) \subset X_j$ for $j= 1, \dots, n$, that realize the local validity of the $\CD^*(K',N)$ condition. In particular we ask that, for $j=1,\dots,n$, every pair of marginals $\mu_0,\mu_1\in \Prob_N^*(\X,\m)$ such that $\supp(\mu_0),\supp(\mu_1)\subseteq X_j$ can be joined by a geodesic in $\Prob_N^*(\X, \m)$ satisfying \eqref{def:CD}. Hence, we choose a $\kappa \in \N$ such that
\[
2^{-\kappa} \diam(\bar B_{2R}(o)) \leq 2^{2-\kappa} R \le \lambda.
\]

\begin{lemma}\label{lemma:Ckappa}
Under the assumptions of Theorem \ref{thm:globloc}, property $\mathsf C(\kappa)$ holds true.
\end{lemma}

\begin{proof}
We fix a geodesic $\Gamma \colon [0, 1] \to \Prob_N^* (\X,  \m)$ such that $\supp(\mu_0),\supp(\mu_1)\subseteq B_R(o )$ and a pair of times $s, t \in [0,1]$, such that $t-s = 2^{-\kappa}$. We consider 
\begin{equation*}
    \hat{\mathsf q}= (e_s,e_t)_\# \Gamma \in \Prob(\X\times \X).
\end{equation*}
It is easy to realize that 
\begin{equation}\label{eq:geoXc}
    \di (x,y) \leq 2^{-\kappa} \diam (\bar B_{2R}(o)) \le \lambda,
\end{equation}
for $\hat{\mathsf q}$-almost every $(x,y)$. Then, for $j = 1, \dots, n$, we define the probability measures $\Gamma_j(s)$ and $\Gamma_j(t)$ by
\begin{equation*}
    \Gamma_j(s):= \frac{1}{\alpha_j}(\p_1)_\# [\hat{\mathsf q}|_{L_j \times \X}] \quad \text{ and } \quad \Gamma_j(t):=\frac{1}{\alpha_j}(\p_2)_\# [\hat{\mathsf q}|_{L_j \times \X}],
\end{equation*}
provided that $\alpha_j := \hat{\mathsf q}(L_j \times \X)> 0$ (otherwise we can define $\Gamma_j(s)$ and $\Gamma_j(t)$ arbitrarily). In the last formula $\p_1,\p_2: \X \times \X \to \X$ denote the projection maps on the first and on the second factor, respectively.  Then $\supp(\Gamma_j(s)) \subseteq \overline{L_j}$ and this, together with \eqref{eq:geoXc}, ensures that
\[
\supp(\Gamma_j(s)) \cup \supp(\Gamma_j(t)) \subseteq \overline{B_\lambda (L_j)} \subseteq X_j.
\]
Moreover, notice that $\Gamma_j(s),\Gamma_j(t)\in \Prob_N^*(\X,\m)$ for every $j$, thus, according to our choice of the sets $X_j$, there exists a geodesic $\hat \Gamma_j$ between $\Gamma_j(s)$ and $\Gamma_j(t) \in \Prob_N^*(\X,\m)$ such that
\begin{equation}\label{eq:CDj}
{S}_{N',\m}\big(\hat\Gamma_j(r)\big) \le \sigma^{(1-r)}_{K', N'}( \theta_j) {S}_{N',\m}(\Gamma_j(s)) + \sigma^{(r)}_{K', N'}( \theta_j) {S}_{N',\m}(\Gamma_j(t)),
\end{equation}
for all $r \in [0, 1]$ and $N' \in [N, 0)$, where
\[
\theta_j := \inf_{\gamma \in \supp(\hat\Gamma_j)} \mathsf d(\gamma(s), \gamma(t)).
\]
Define then the curve $\hat \Gamma \colon [0, 1] \to \Prob_N^* (\X, \m)$ by setting for any $r \in [0, 1]$
\[
 \hat \Gamma(r) := \sum_{j=1}^n \alpha_j \hat\Gamma_j( r).
\]
We observe that $\{\hat \Gamma(r)\}_{r \in [0, 1]}$ is a geodesic between $\Gamma(s) := \sum_{j=1}^n \alpha_j \Gamma_j(s)$ and $\Gamma(t) := \sum_{j=1}^n \alpha_j \Gamma_j(t)$ and so, as a consequence of Proposition \ref{prop:ambrosiogigli} we can actually conclude that 
\begin{equation*}
    \hat\Gamma (r)=\Gamma(s + r(t-s)) \quad \forall r\in[0,1].
\end{equation*}
It is easy to realize that this in particular implies that $\theta_j\geq \theta^\kappa$ for every $j$, and then, keeping in mind \eqref{eq:CDj}, we conclude that 
\begin{equation}\label{eq:thetakappa}
    {S}_{N',\m}\big(\hat\Gamma_j(r)\big) \le \sigma^{(1-r)}_{K', N'}( \theta^\kappa) {S}_{N',\m}(\Gamma_j(s)) + \sigma^{(r)}_{K', N'}( \theta^\kappa) {S}_{N',\m}(\Gamma_j(t)),
\end{equation}
for every $j$. On the other hand, since $\Gamma_j (s)$ are mutually singular for $j = 1, \dots, n$, it is possible to apply Proposition \ref{prop:orthogonality} and conclude that $\hat \Gamma(r)$ are mutually singular for $r\in[0,1)$ and $j = 1, \dots, n$. In particular, the fact that $\Gamma_j (s)$ and $\hat \Gamma_j( r)$ are mutually singular ensures that for every  $N' \in [N, 0)$ it holds
\begin{equation}\label{eq:sumr}
    S_{N', \m} \big(\hat \Gamma(r)\big) = \sum_{j=1}^n \alpha_j^{1-\frac{1}{N'}} S_{N', \m} \big( \hat \Gamma_j(r)\big) \quad \forall r\in[0,1)
\end{equation}
and
\begin{equation}\label{eq:sums}
    S_{N', \m} \big( \Gamma(s)\big) = \sum_{j=1}^n \alpha_j^{1-\frac{1}{N'}} S_{N', \m} \big( \Gamma_j(s)\big).
\end{equation}
On the other hand, the $\Gamma_j(t)$ are not necessarily mutually singular for $j=1, \dots, n$, and so
\begin{equation}\label{eq:sumt}
    S_{N', \m} \big( \Gamma(t)\big) \ge \sum_{j=1}^n \alpha_j^{1-\frac{1}{N'}} S_{N', \m} \big( \Gamma_j(t)\big).
\end{equation}
At this point, summing up for $j = 1, \dots, n$ the inequality \eqref{eq:thetakappa} and making use of \eqref{eq:sumr}, \eqref{eq:sums} and \eqref{eq:sumt}, we obtain
\[
\begin{split}
S_{N', \m} \big( \Gamma(s + r(t-s))\big) &= S_{N', \m} \big(\hat \Gamma(r)\big)=\sum_{j=1}^n \alpha_j^{1-\frac{1}{N'}} {S}_{N',\m}\big(\hat\Gamma_j(r)\big)\\ &\le \sigma^{(1-r)}_{K', N'}( \theta^\kappa) \sum_{j=1}^n \alpha_j^{1-\frac{1}{N'}} {S}_{N',\m}(\Gamma_j(s)) + \sigma^{(r)}_{K', N'}( \theta^\kappa) \sum_{j=1}^n \alpha_j^{1-\frac{1}{N'}} {S}_{N',\m}(\Gamma_j(t))\\
& \le \sigma^{(1-r)}_{K', N'}( \theta^\kappa) S_{N', \m} \big( \Gamma(s)\big) + \sigma^{(r)}_{K', N'}( \theta^\kappa)  S_{N', \m} \big( \Gamma(t)\big)
\end{split}
\]
for all $N' \in [N, 0)$, proving property $\mathsf C(\kappa)$.
\end{proof}

Using these results, the proof of Theorem \ref{thm:globloc} is quite straightforward.
\begin{proof}[Proof of Theorem \ref{thm:globloc}]
Let us fix two probability measures $\mu_0, \mu_1 \in \Prob_N^* (\X, \m)$ such that $$\supp(\mu_0),\supp(\mu_1)\subseteq B_R(o ).$$By assumption, there exists a geodesic $\Gamma$ with domain in $\Prob_N^* (\X, \m)$ connecting them, i.e. $\Gamma(0) = \mu_0$ and $\Gamma(1) = \mu_1$. Now, by Lemma \ref{lemma:Ckappa}, property $\mathsf C(\kappa)$ is satisfied, while Lemma \ref{lemma:k-1} ensures that this implies that $\mathsf C(k)$ holds also for all $k = \kappa-1, \kappa-2, \dots, 0$. In particular, property $\mathsf C(0)$ states that the geodesic $\Gamma$ is such that
\[
S_{N', \m}(\Gamma(t)) \le \sigma_{K', N'}^{(1-t)}(\theta^0)S_{N', \m}(\mu_0) + \sigma_{K', N'}^{(t)}(\theta^0)S_{N', \m}(\mu_1)
\]
for all $N' \in [N, 0)$, where
\begin{equation*}
    \theta^0= \inf_{\gamma \in \supp(\Gamma)} \mathsf d(\gamma(0), \gamma(1)).
\end{equation*}
On the other hand, it is obvious that
\[
\theta^0\geq \theta = \inf_{x_0 \in \mathcal S_0, x_1 \in \mathcal S_1} \di(x_0, x_1),
\]
where $\mathcal S_0 = \supp(\mu_0)$ and $\mathcal S_1 = \supp(\mu_1)$. As a consequence we can conclude that
\[
S_{N', \m}(\Gamma(t)) \le \sigma_{K', N'}^{(1-t)}(\theta)S_{N', \m}(\mu_0) + \sigma_{K', N'}^{(t)}(\theta)S_{N', \m}(\mu_1).
\]
Thanks to the arbitrariness of $K' < K$ and of the metric ball $B_R(o)$, it is possible to apply Proposition \ref{prop:equivalent} and show the validity of the condition $\CD^\ast(K-, N)$ globally in $(\X, \di, \m)$.
\end{proof}

We conclude by noticing that combining Theorem \ref{thm:globloc} with \ref{prop:positivecurv}, we obtain the local-to-global property for the $\CD^*(K,N)$ condition, when the curvature parameter $K$ is non-negative.

\begin{corollary}
Let $K, N \in \R$ with $N<0\leq K$ and let $(\X,\mathsf{d},\m)$ be a locally compact, essentially non-branching metric measure space such that $\Prob_N^* (\X,\m)$ is a geodesic space. If $(\X,\mathsf{d},\m)$ satisfies the condition $\CD^\ast(K, N)$ locally, then it satisfies the condition $\CD^\ast(K, N)$ globally.
\end{corollary}

\bibliographystyle{abbrv} 
\bibliography{maps_loctoglo}

\end{document}